\documentclass[12pt]{amsart}
\usepackage{amssymb}
\usepackage{amsmath}
\numberwithin{equation}{section}
\usepackage{color}
\usepackage{graphicx}
\usepackage{epstopdf}

\newtheorem{thm}{Theorem}[section]
\newtheorem{lemma}[thm]{Lemma}
\newtheorem{cor}[thm]{Corollary}

\theoremstyle{definition}

\newtheorem{defn}[thm]{Definition}

\theoremstyle{remark}
\newtheorem{remark}[thm]{Remark}

\theoremstyle{rem}

\newcommand\bq{\begin{equation}}
\newcommand\eq{\end{equation}}
\newcommand\beq{\begin{eqnarray*}}
\newcommand\eeq{\end{eqnarray*}}
\newcommand\ben{\begin{enumerate}}
\newcommand\een{\end{enumerate}}
\newcommand\bit{\begin{itemize}}
\newcommand\eit{\end{itemize}}

\newcommand\Fix{{\rm FIX}}
\newcommand\Col{{\rm COL}}
\newcommand\fexc{{\rm fexc}}
\newcommand\fdes{{\rm fdes}}
\newcommand\fmaj{{\rm fmaj}}
\newcommand\typeB{{\rm B}}
\newcommand\col{{\rm col}}
\newcommand\des{{\rm des}}
\newcommand\exc{{\rm exc}}

\newcommand\maj{{\rm maj}}

\newcommand\Des{{\rm DES}}
\newcommand\Exd{{\rm DEX}}
\newcommand\Exc{{\rm EXC}}
\newcommand\fix{{\rm fix}}

\newcommand\Exp{{\rm Exp}}
\newcommand\wt{{\rm wt}}
\newcommand\wtt{{\rm \widetilde{wt}}}

\newcommand\Com{{\rm Com}}

\def\Lambda{{\bf ps}}
\def\bar{\overline}

\def\xx{{\mathbf x}}

\keywords{Signed permutations, Colored permutations, Permutation statistics, Eulerian polynomials, Quasisymmetric functions, Symmetric functions}

\begin{document}

\title[Eulerian qsym. functions for wreath product groups]{Eulerian quasisymmetric functions for the type B Coxeter group and other wreath product groups}
\date{September 13, 2011 \\
\text{ }\text{ }\text{ } 2010 \textit{Mathematics Subject Classifications}: 05A15, 05E05}
\author[M. Hyatt]{Matthew Hyatt}
\address{Department of Mathematics, University of California, San Diego, La Jolla, CA 92093}
\email{matthewdhyatt@gmail.com}

\begin{abstract}
Eulerian quasisymmetric functions were introduced by Shareshian and Wachs in order to obtain a $q$-analog of Euler's exponential generating function formula for the Eulerian numbers \cite{sw}. They are defined via the symmetric group, and applying the stable and nonstable principal specializations yields formulas for joint distributions of permutation statistics. We consider the wreath product of the cyclic group with the symmetric group, also known as the group of colored permutations. We use this group to introduce \textit{colored Eulerian quasisymmetric functions}, which are a generalization of Eulerian quasisymmetric functions. We derive a formula for the generating function of these colored Eulerian quasisymmetric functions, which reduces to a formula of Shareshian and Wachs for the Eulerian quasisymmetric functions. We show that applying the stable and nonstable principal specializations yields formulas for joint distributions of colored permutation statistics, which generalize the Shareshian-Wachs $q$-analog of Euler's formula, formulas of Foata and Han, and a formula of Chow and Gessel.
\end{abstract}

\maketitle

\vbox{\tableofcontents}

\section{Introduction}

A permutation statistic $f$, is a map $f:S_n\rightarrow\mathbb{N}$ where $S_n$ is the symmetric group and $\mathbb{N}$ is the set of nonnegative integers. Two well known permutation statistics are descent number and excedance number, denoted des and exc respectively. Given any  $\sigma \in S_n$, we define the descent set and excedance set of $\sigma$ by
\[\Des(\sigma):=\left\{i\in \left[n-1\right]:\sigma(i)>\sigma(i+1)\right\}\]
and
\[\Exc(\sigma):=\left\{i\in \left[n-1\right]:\sigma(i)>i\right\}\]
where $[n]:=\{1,2,\dots,n\}$. The descent number and excedance number are then defined by
\[\des(\sigma):=|\Des(\sigma)|\hspace{.5cm}\text{and}\hspace{.5cm}
\exc:=|\Exc(\sigma)|.\]
MacMahon \cite{mcmahon} showed that these two statistics are equidistributed, that is
\[\sum_{\pi\in S_n}t^{\text{des}(\pi)}
=\sum_{\pi\in S_n}t^{\text{exc}(\pi)}.\]
These polynomials, which we denote by $A_n(t)$, are known as Eulerian polynomials. Thus any other permutation statistic that is equidistributed with des or exc is called an Eulerian permutation statistic. The Eulerian polynomials have arisen in many areas of mathematics, as well as in computer science, physics, and biology. Euler proved (see \cite{knuth}) the following formula for the exponential generating function of these numbers
\begin{equation}
\sum_{n\geq 0}A_n(t)\frac{z^n}{n!}=\frac{1-t}{e^{(t-1)z}-t}.
\label{intro1}
\end{equation}

There has been much work done in studying joint distributions of statistics on the symmetric group and colored permutation groups (see e.g. \cite{abr},\cite{cg},\cite{fh2},\cite{fh},\cite{gr},\cite{reiner},\cite{sw1},\cite{sw}). A particular example that was the first to pair the permutation statistics maj and exc is the following formula discovered by Shareshian and Wachs \cite[Corollary 1.3]{sw}
\begin{equation}
\sum_{\substack{n\geq 0 \\ \pi\in S_n}}q^{\text{maj}(\pi)}t^{\text{exc}(\pi)}r^{\text{fix}(\pi)}
\frac{z^n}{\left[n\right]_q!}
=\frac{(1-tq)\exp_q(rz)}{\exp_q(tqz)-tq\exp_q(z)},
\label{intro2}
\end{equation}
where $\left[n\right]_q:=1+q+...+q^{n-1}$ is the $q$-analog of $n$, $\left[n\right]_q!:=\prod_{j=1}^n\left[j\right]_q$ is the $q$-analog of $n!$, $\exp_q(z):=\sum_{i\geq 0}\left(z^i/\left[i\right]_q!\right)$ is one of two $q$-exponential functions, maj denotes the major index defined by
\[\maj(\sigma):=\sum_{i\in\Des(\sigma)}i,\]
and fix denotes the number of fixed points of a permutation defined by
\[\fix(\sigma):=|\left\{i\in [n]:\sigma(i)=i\right\}|.\]
One can recover $\eqref{intro1}$ from $\eqref{intro2}$ by setting $q=r=1$.

To prove $\eqref{intro2}$, they introduce a family of quasisymmetric functions $Q_{n,j,k}(\textbf{x})$, called Eulerian quasisymmetric functions, where $\textbf{x}$ denotes the infinite set of variables $\left\{x_1,x_2,...\right\}$. They compute the following generating function \cite[Theorem 1.2]{sw}

\begin{equation}
\sum_{n,j,k\geq 0}Q_{n,j,k}(\textbf{x})t^jr^kz^n
=\frac{(1-t)H(rz)}{H(tz)-tH(z)},
\label{eq10}
\end{equation}
where $H(z):=\sum_{i\geq 0}h_i(\textbf{x})z^i$ and $h_i$ is the complete homogeneous symmetric function of degree $i$. The Eulerian quasisymmetric functions are constructed (and so named) because applying the stable principal specialization to $\eqref{eq10}$ yields the formula for the joint distribution above. The Eulerian quasisymmetric functions are also quite interesting in their own right, and many other properties are investigated in \cite{sw}.

In this paper, we generalize Eulerian quasisymmetric functions and the formula $\eqref{eq10}$, such that applying specializations yields colored analogs of \eqref{intro2}. Let $\mathbb{P}$ denote the set of positive integers. For $N\in\mathbb{P}$ we let $C_N$ denote the cyclic group of order $N$, and consider the wreath product $C_N\wr S_n$, also called the colored permutation group (see also Section 2). In Section 2 we use $C_N\wr S_n$ to define a family of quasisymmetric functions which we call colored Eulerian quasisymmetric functions, denoted $Q_{n,j,\vec{\alpha},\vec{\beta}}(\textbf{x})$ where $n,j\in\mathbb{N}$, $\vec{\alpha}\in\mathbb{N}^{N}$, and $\vec{\beta}\in\mathbb{N}^{N-1}$. They are a generalization of $Q_{n,j,k}(\textbf{x})$ in the sense that if we set $\vec{\alpha}=(k,0,0,...,0)$, then $Q_{n,j,\vec{\alpha},\vec{0}}(\textbf{x})=Q_{n,j,k}(\textbf{x})$.

In \cite{sw}, the authors prove $\eqref{eq10}$ by developing nontrivial extensions of techniques first developed by Gessel and Reutenauer in \cite{gr}. Here we develop further nontrivial extensions of those techniques developed in \cite{sw}, in order to compute the following generating function.

\begin{thm} \label{thm 1.1}

Fix $N\in\mathbb{P}$ and let $r^{\vec{\alpha}}=r_0^{\alpha_0}\cdots r_{N-1}^{\alpha_{N-1}}$ and $s^{\vec{\beta}}=s_1^{\beta_1}\cdots s_{N-1}^{\beta_{N-1}}$. Then
\vspace{.25cm}
\[\sum_{\substack{n,j\geq 0 \\ \vec{\alpha}\in\mathbb{N}^N \\
\vec{\beta}\in\mathbb{N}^{N-1}}}
Q_{n,j,\vec{\alpha},\vec{\beta}}(\xx)z^nt^jr^{\vec{\alpha}}s^{\vec{\beta}}\]
\[=\frac{H(r_0z)(1-t)\left(\prod_{m=1}^{N-1}E(-s_mz)H(r_ms_mz)\right)}
{\left(1+\sum_{m=1}^{N-1}s_m\right)H(tz)-\left(t+\sum_{m=1}^{N-1}s_m\right)H(z)},\]

\vspace{.25cm}
where $E(z):=\sum_{i\geq 0}e_i(\xx)z^i$ and $e_i$ is the elementary symmetric function of degree $i$.
\end{thm}

\vspace{.25cm}
Much of this paper, specifically Sections 3 and 4, will be devoted to proving Theorem~\ref{thm 1.1}. The next two corollaries state immediate and interesting consequences of this theorem.

\begin{cor} \label{cor 1.2}
The quasisymmetric function $Q_{n,j,\vec{\alpha},\vec{\beta}}(\xx)$ is actually symmetric.
\end{cor}

\begin{cor} \label{cor 1.3}
Let $N\geq 2$. Set $r_k=1$ for $0\leq k\leq N-1$ and set $s_m=\omega^m$ for $1\leq m\leq N-1$ where $\omega$ is a primitive $N^{\text{th}}$ root of unity. Then the right hand side of Theorem~\ref{thm 1.1} is equal to 1. That is
\[\sum_{\substack{n,j\geq 0 \\ \vec{\alpha}\in\mathbb{N}^N \\
\vec{\beta}\in\mathbb{N}^{N-1}}}
Q_{n,j,\vec{\alpha},\vec{\beta}}(\xx)z^nt^j
\left(\prod_{m=1}^{N-1}\omega^{m\cdot\beta_m}\right)=1.\]
\end{cor}

\vspace{.25cm}
We note that the definitions of $\des$ and $\maj$ given above can be applied in an obvious way to words of length $n$ over any ordered alphabet, and the definition of $\exc$ can be applied to words over an ordered alphabet that includes $[n]$. Let $\vec{\text{fix}}(\pi)=(\text{fix}_0(\pi),...,\text{fix}_{N-1}(\pi))\in\mathbb{N}^N$ and $\vec{\text{col}}(\pi)=(\text{col}_1(\pi),...,\text{col}_{N-1}(\pi))\in\mathbb{N}^{N-1}$ where $\pi\in C_N\wr S_n$ and $\text{fix}_i(\pi)$ is the number of fixed points of $\pi$ of color $i$, and $\text{col}_i(\pi)$ is the number of letters of $\pi$ of color $i$.  See the beginning of Section 2 for complete definitions of the various colored permutation statistics used.

By applying the stable and nonstable principal specializations respectively to Theorem \ref{thm 1.1}, we obtain the following joint distribution formulas.

\begin{thm} \label{thm 1.4}
\[\sum_{\substack{n\geq 0 \\ \pi\in C_N\wr S_n}}
\frac{z^n}{\left[n\right]_q!}
t^{\exc(\pi)}r^{\vec{\fix}(\pi)}s^{\vec{\col}(\pi)}
q^{\maj(\pi)}\]
\[=\frac{\exp_q(r_0z)(1-tq)\left(\prod_{m=1}^{N-1}\Exp_q(-s_mz)\exp_q(r_ms_mz)\right)}{\left(1+\sum_{m=1}^{N-1}s_m\right)\exp_q(tqz)-\left(tq+\sum_{m=1}^{N-1}s_m\right)\exp_q(z)},\]

\vspace{.25cm}
where $\Exp_q(z):=\sum_{i\geq 0}\left(z^iq^{{i\choose 2}}/\left[i\right]_q!\right)$.
\end{thm}

\begin{thm} \label{thm 1.5}
\[\sum_{\substack{n\geq 0 \\ \pi\in C_N\wr S_n}}\frac{z^n}{(p;q)_{n+1}}
t^{\exc(\pi)}r^{\vec{\fix}(\pi)}s^{\vec{\col}(\pi)}
q^{\maj(\pi)}p^{\des^*(\pi)}\]
\[\hspace{-1.5cm}=\sum_{l\geq 0}\frac{p^l(1-tq)(z;q)_l(tqz;q)_l\left(\prod_{m=1}^{N-1}(s_mz;q)_l\right)}
{\left(\prod_{m=1}^{N-1}(r_ms_mz;q)_l\right)(r_0z;q)_{l+1}\left[\left(1+\sum_{m=1}^{N-1}s_m\right)(z;q)_l-\left(tq+\sum_{m=1}^{N-1}s_m\right)(tqz;q)_l\right]},\]

\vspace{.25cm}
where $(a;q)_n$ is defined as
\[(a;q)_n:=\left\{\begin{array}{lr}
1 & \text{ if }n=0\\
(1-a)(1-aq)...(1-aq^{n-1}) & \text{ if }n\geq 1\end{array}\right..\]
\end{thm}

\vspace{.25cm}
Thus Theorem~\ref{thm 1.5} is a formula for the joint distribution of $2N+2$ colored permutation statistics. The proofs of Theorems 1.4 and 1.5 will be given at the end of Section 2.

If we make the same variable evaluation as in Corollary 1.3, we obtain the following interesting consequences of Theorem 1.5.

\begin{cor} \label{cor 1.6}
Let $N\geq 2$ and let $\omega$ be a primitive $N^{\text{th}}$ root of unity, then
\[\sum_{\substack{n\geq 0 \\ \pi\in C_N\wr S_n}}
\frac{z^n}{(p;q)_{n+1}}
t^{\exc(\pi)}q^{\maj(\pi)}p^{\des^*(\pi)}
\left(\prod_{m=1}^{N-1}\omega^{m\cdot\col_m(\pi)}\right)\]
\[=\sum_{l\geq 0}\frac{p^l}{1-zq^l}
=\sum_{n,l\geq 0}z^nq^{nl}p^l.\]
Taking the coefficient of $z^n$ on both sides we have
\[\sum_{\pi\in C_N\wr S_n}
t^{\exc(\pi)}q^{\maj(\pi)}p^{\des^*(\pi)}\left(\prod_{m=1}^{N-1}\omega^{m\cdot\text{\col}_m(\pi)}\right)
=(p;q)_n.\]
\end{cor}

\vspace{.25cm}
Before moving on to the next section, we would like to take note of some special cases. First we note that in the case $N=1$, we have $C_1\wr S_n\cong S_n$, and the colored permutation statistics exc, fix, maj, $\des^*$ become permutation statistics exc, fix, maj, des respectively. Theorem \ref{thm 1.4} reduces to a formula of Shareshian and Wachs \cite{sw}, and Theorem \ref{thm 1.5} reduces to a formula of Foata and Han \cite{fh2}.

The case $N=2$ is also of particular interest since $C_2\wr S_n\cong B_n$, where $B_n$ denotes the type B Coxeter group, also called the hyperoctahedral group. The elements of $B_n$ are often called signed permutations because they can be described as elements of $S_n$ viewed as words, but where each letter is signed. There are various candidates for type B analogs of the descent number, excedance number and major index that have been proposed in the literature. We mention a few here. The type B descent number has a connection to Coxeter length on $B_n$ (see \cite{bb}), and is defined for $\pi\in B_n$ by 
\[\des_{\typeB}(\pi):=\des(\pi)+ \chi(\pi_1 < 0),\] 
where $\chi$ is the characteristic function defined by $\chi(P)=1$ if $P$ is a true statement and is $0$ otherwise, and $\des(\pi)$ is computed with respect to the natural order on $\mathbb{Z}$. The flag descent number was proposed by Adin, Brenti and Roichman \cite{abr} as a partner to the flag major index, introduced  by Adin and Roichman \cite{ar}, in a type B analog of a well known formula of Carlitz \cite{carlitz} for a $q$-analog of the Eulerian polynomials.  
These statistics are defined for $\pi \in B_n$ by 
\[\fdes(\pi) := 2\des(\pi)+ \chi(\pi_1 < 0),\]
\[\fmaj(\pi):=2\maj(\pi)+\mathrm{neg}(\pi),\]
where
\[\mathrm{neg}(\pi):=\left|\left\{i\in[n]:\pi_i<0\right\}\right|.\]
In \cite{ar} it is shown that the flag major index is equidistributed with Coxeter length on $B_n$ making it a true type B Mahonian statistic. We note that there are two different flag major index statistics, one computed using the natural order on $\mathbb{Z}$, and one using the order $-1<-2<...<-n<1<2<...n$. While these statistics are different, both are equidistributed with Coxeter length on $B_n$. More recently Chow and Gessel \cite{cg} obtained a different type B analog of Carlitz's formula, which involves the type B Euler-Mahonian pair $(\des_{\text{B}},\fmaj)$. An excedance statistic for the type B Coxeter group that is equidistributed with $\fdes$ was introduced by Foata and Han \cite{fh}. It is called the flag excedance number and it is defined by
\[\fexc(\pi) = 2 \exc(\pi) + \mathrm{neg}(\pi).\]

By making the change of variables $q\mapsto q^2$, $t\mapsto t^2$, and $s_1\mapsto qts_1$ in Theorems \ref{thm 1.4} and \ref{thm 1.5} in the case $N=2$, the colored permutation statistics maj and exc become fmaj and fexc respectively. We note that in this case Theorems \ref{thm 1.4} and \ref{thm 1.5} are similar, but not equivalent, to formulas discovered by Foata and Han \cite[Corollary 1.2, Theorem 1.1]{fh}. One difference stems from the fact that different orders are used to define the descent set of a signed permutation. Another difference is that $\des^*$ appears in Theorem \ref{thm 1.5}, while $\fdes$ appears in \cite[Theorem 1.1]{fh}.




%

It is also important to note that $\des_{\typeB}$ is equidistributed with $\des^*$ when $N=2$. This can be seen from the following bijection $\varphi$ on $B_n$. Given $\pi\in B_n$, $\varphi(\pi)$ is obtained by re-writing each block of negative letters in reverse. For example
\[\varphi([\begin{array}{ccccccccc}
-2 & -6 & -3 & 1 & 9 & -8 & -5 & -4 & -7 \end{array}])\]
\[=[\begin{array}{ccccccccc}
-3 & -6 & -2 & 1 & 9 & -7 & -4 & -5 & -8 \end{array}].\]
Clearly $\des^*(\varphi(\pi))=\des_{\typeB}(\pi)$. And while $\varphi$ does not preserve negative fixed points, it does preserve the number of negative letters and the flag excedance number. It preserves the number positive fixed points, denoted $\fix^+$, defined by
\[\fix^+(\pi)=|\left\{i\in[n]:\pi_i=i\right\}|.\]
It also changes $\fmaj$ to the other flag major index which, as noted above, is also equidistributed with Coxeter length on $B_n$. Thus we have the following formula
\begin{cor} \label{cor 1.9}
\[\sum_{\substack{n\geq 0 \\ \pi\in B_n}}\frac{z^n}{(p;q^2)_{n+1}}
q^{\fmaj(\pi)}t^{\fexc(\pi)}p^{\des_{\typeB}(\pi)}r^{\fix^+(\pi)}s^{\mathrm{neg}(\pi)}\]
\[=\sum_{k\geq0}\frac{p^k(1-t^2q^2)(z;q^2)_k(t^2q^2z;q^2)_k}
{(rz;q^2)_{k+1}\left[(1+sqt)(z;q^2)_k-(t^2q^2+sqt)(t^2q^2z;q^2)_k\right]}.\]
\end{cor}
By setting $t=r=s=1$ and extracting the coefficient of $z^n$ from both sides, Corollary \ref{cor 1.9} reduces to the following Gessel-Chow type B analog of Carlitz's formula. 

\begin{cor}[{\cite[Lemma 5.2]{cg}}] \label{cor 1.10}

\[\sum_{\pi\in B_n}q^{\fmaj(\pi)}p^{\des_{\typeB}(\pi)}=(p;q^2)_{n+1}\sum_{k\geq 0}p^k[2k+1]_q^n.\]

\end{cor}

\begin{proof}

From Corollary \ref{cor 1.9} we have
\[\sum_{\substack{n\geq 0 \\ \pi\in B_n}}\frac{z^n}{(p;q^2)_{n+1}}q^{\fmaj(\pi)}p^{\des_{\typeB}(\pi)}\]
\[=\sum_{k\geq 0}\frac{p^k(1-q^2)(z;q^2)_k(q^2z;q^2)_k}
{(z;q^2)_{k+1}[(1+q)(z;q^2)_k-(q^2+q)(q^2z;q^2)_k}.\]

Note that $(zq^2;q^2)_k=(z;q^2)_{k+1}/(1-z)$, thus
\[\sum_{\substack{n\geq 0 \\ \pi\in B_n}}\frac{z^n}{(p;q^2)_{n+1}}q^{\fmaj(\pi)}p^{\des_{\typeB}(\pi)}\]
\[=\sum_{k\geq 0}\frac{p^k(1-q)(1+q)(z;q^2)_k}
{(1-z)[(1+q)(z;q^2)_k-(q^2+q)(q^2z;q^2)_k}\]
\[=\sum_{k\geq 0}\frac{p^k(1-q)(z;q^2)_k}
{(1-z)(z;q^2)_k-q(z;q^2)_{k+1}}
=\sum_{k\geq 0}\frac{p^k(1-q)}
{(1-z)-q(1-zq^{2k})}\]
\[=\sum_{k\geq 0}\frac{p^k(1-q)}
{(1-q)-z(1-q^{2k+1})}
=\sum_{k\geq 0}\frac{p^k}{1-z[2k+1]_q}\]
\[=\sum_{n,k\geq 0}z^n[2k+1]_q^np^k.\]

Now extract the coefficient of $z^n$ from both sides to obtain the desired formula.

\end{proof}

\vspace{.5cm}
\section{Colored Eulerian quasisymmetric functions and the colored permutation group}

Consider the following ordered alphabet of the $N$-colored integers from 1 to $n$
\[\hspace{-.8cm}\mathcal{E}:=\left\{1^{N-1}<2^{N-1}...<n^{N-1}
<1^{N-2}<2^{N-2}...<n^{N-2}<...
<1^0<2^0<...<n^0\right\}.\]

If $\pi$ is a word over $\mathcal{E}$, we use $\pi(i)=\pi_i$ to denote the $i^{\text{th}}$ letter of the word $\pi$. We let $\left|\pi_i\right|$ denote the positive integer obtained by removing the superscript, and let $\epsilon_i\in\left\{0,1,...,N-1\right\}$ denote the superscript, or color, of the $i^{\text{th}}$ letter of the word. If $\pi$ is word of length $n$ over $\mathcal{E}$, we denote by $\left|\pi\right|$ the word
\[\left|\pi\right|:=\left|\pi_1\right|\left|\pi_2\right|...\left|\pi_n\right|.\]
Consider the wreath product $C_N\wr S_n$, where $C_N$ is the cyclic group of order $N$ and $S_n$ is the symmetric group of order $n!$. The group $C_N\wr S_n$ is also known as the colored permutation group, and can be viewed as the set of words over $\mathcal{E}$ defined by $\pi\in C_N\wr S_n$ iff $\left|\pi\right|\in S_n$.

For example, if we write a word in two-line notation as
\[\pi=\begin{bmatrix} 1 & 2 & 3 & 4 & 5 & \\
3^2 & 5^0 & 4^1 & 1^2 & 2^1\end{bmatrix}\in C_3\wr S_5,\]
then
\[\left|\pi\right|=\begin{bmatrix} 1 & 2 & 3 & 4 & 5 & \\
3 & 5 & 4 & 1 & 2\end{bmatrix}\in S_5.\]
In particular, $\left|\pi_3\right|=4$ and $\epsilon_3=1$ so that $\pi_3=\left|\pi_3\right|^{\epsilon_3}=4^1$.

We can also write colored permutations in cycle notation using the convention that $j^{\epsilon_j}$ follows $i^{\epsilon_i}$ means that $\pi_i=j^{\epsilon_j}$. It is easy to see that a colored permutation decomposes into a product of disjoint cycles. Continuing with the previous example, we can write it in cycle notation as
\[\pi=(1^2,3^2,4^1)(2^1,5^0).\]
Once we define the color vector below, we will be in a position to define a useful refinement of the notion of the cycle type of a colored permutation.

Our next objective is to define all of the relevant colored permutation statistics for $C_N\wr S_n$ that we will use. Using the order on the alphabet $\mathcal{E}$ given above, we define the descent set of $\pi$, denoted $\text{DES}(\pi)$, by 
\[\Des(\pi):=\left\{i\in [n-1]:\pi_i>\pi_{i+1}\right\}.\]
We define descent number and major index of a colored permutation, denoted des and maj respectively, by 
\[\text{des}(\pi):=\left|\text{DES}(\pi)\right|\] 
and 
\[\text{maj}(\pi):=\sum_{i\in\text{DES}(\pi)}i.\]

Similar to the descent set is the starred descent set of a colored permutation denoted by $\Des^*(\pi)$, and defined by
\[\text{DES}^*(\pi):=\left\{\begin{array}{lr}
\text{DES}(\pi) & \text{ if }\epsilon_1=0\\
\text{DES}(\pi)\cup\left\{0\right\} & \text{ if } \epsilon_1>0\end{array}\right..\]
Thus $\text{DES}^*(\pi)\subseteq
\left[n-1\right]\cup\left\{0\right\}$, and we define the starred descent number $\des^*(\pi)$ by
\[\text{des}^*(\pi):=\left|\text{DES}^*(\pi)\right|.\]
We note that when identifying the type B Coxeter group with $C_2\wr S_n$, the definition for the type B descent number uses the order
\[n^1<(n-1)^1<...<1^1<1^0<2^0<...<n^0,\]
where the letters with superscript 1 are identified with negative integers. So while $\des_{\typeB}$ does not agree with $\des^*$ for $N=2$, the simple bijection $\gamma:C_2\wr S_n\rightarrow C_2\wr S_n$ defined by rewriting each string of negative integers in reverse order shows that $\des_{\typeB}$ and $\des^*$ are equidistributed. For example if $\pi=2^0, 1^1, 5^1, 3^0, 8^0, 7^1, 4^1, 6^1$ then $\gamma(\pi)=2^0, 5^1, 1^1, 3^0, 8^0, 6^1, 4^1, 7^1$ and $\des_{\typeB}(\pi)=\des^*(\gamma(\pi))$. We also note that $\des^*$ agrees with the usual descent number for $C_1\wr S_n\cong S_n$.

We define the excedance set of a colored permutation $\pi$, denoted $\Exc(\pi)$, by 
\[\Exc(\pi):=\left\{i\in [n]:\pi_i>i^0\right\}.\]
We define the excedance number $\exc(\pi)$ by
\[\exc(\pi):=|\Exc(\pi)|.\]

Given a nonnegative integer $k$ such that $0\leq k\leq N-1$, we define the $k^{\text{th}}$ color fixed point set of $\pi\in C_N\wr S_n$, denoted $\Fix_k(\pi)$, by 
\[\Fix_k(\pi):=\left\{i\in [n]:\pi_i=i^k\right\}.\]
We define the $k^{\text{th}}$ color fixed point number $\fix_k(\pi)$ by
\[\fix_k(\pi):=|\Fix_k(\pi)|.\]
And we define the fixed point vector $\vec{\text{fix}}(\pi)\in\mathbb{N}^N$ by
\[\vec{\text{fix}}(\pi):=(\text{fix}_0(\pi),\text{fix}_1(\pi),
...,\text{fix}_{N-1}(\pi)).\]

Our last colored permutation statistic is the color vector. Given $m$ such that $1\leq m\leq N-1$, we define the $m^{\text{th}}$ color set of $\pi\in C_N\wr S_n$, denoted $\Col_m(\pi)$, by 
\[\Col_m(\pi):=\left\{i\in [n]:\epsilon_i=m\right\}.\]
We define the $m^{\text{th}}$ color number $\col_m(\pi)$ by
\[\col_m(\pi):=|\Col_m(\pi)|.\]
And we define the color vector $\vec{\text{col}}(\pi)\in\mathbb{N}^{N-1}$ by
\[\vec{\text{col}}(\pi):=(\text{col}_1(\pi),\text{col}_2(\pi),
...,\text{col}_{N-1}(\pi)).\]

For example if
\[\pi=\begin{bmatrix} 1 & 2 & 3 & 4 & 5 & 6 & 7 & 8\\
1^2 & 4^0 & 8^1 & 6^0 & 5^2 & 3^2 & 7^0 & 2^1 \end{bmatrix}\in C_3\wr S_8,\]
then $\Des(\pi)=\left\{2,4,5,7\right\}$, $\des(\pi)=4$, $\des^*(\pi)=5$, $\exc(\pi)=2$, $\vec{\fix}(\pi)=(1,0,2)$, and $\vec{\col}(\pi)=(2,3)$.

Now we define the cv-cycle type (short for color vector cycle type) of a colored permutation $\pi\in C_N\wr S_n$. As noted above, $\pi$ decomposes into a product of disjoint cycles. Let $\lambda=(\lambda_1\geq...\geq\lambda_k)$ be a partition of $n$. Let $\vec{\beta^1},...,\vec{\beta^k}$ be a sequence of vectors in $\mathbb{N}^{N-1}$ with each $\left|\vec{\beta^i}\right|\leq\lambda_i$, where the absolute value of a vector $\vec{\beta}\in\mathbb{N}^{N-1}$ is the sum of its components, i.e. $|\vec{\beta}|:=\beta_1+\beta_2+...+\beta_{N-1}$. Consider the multiset of pairs
\[\check{\lambda}=\left\{(\lambda_1,\vec{\beta^1}),(\lambda_2,\vec{\beta^2}),...,(\lambda_k,\vec{\beta^k})\right\}.\]
We say that $\pi$ has cv-cycle type $\check{\lambda}(\pi)=\check{\lambda}$ if each pair $(\lambda_i,\vec{\beta^i})$ corresponds to exactly one cycle of length $\lambda_i$ with color vector $\vec{\beta^i}$ in the decomposition of $\pi$. Note that $\vec{\text{col}}(\pi)=\vec{\beta^1}+...+\vec{\beta^k}$ using component wise addition. Consider the following example in $C_3\wr S_9$, let
\[\pi=\begin{bmatrix} 1 & 2 & 3 & 4 & 5 & 6 & 7 & 8 & 9 \\
3^2 & 4^1 & 1^1 & 2^2 & 8^0 & 9^2 & 5^1 & 7^1 & 6^0 \end{bmatrix}\]
\[=(1^1, 3^2) (2^2, 4^1) (6^0, 9^2) (5^1, 8^0, 7^1),\]
so $\check{\lambda}(\pi)=\left\{(3,(2,0)), (2,(1,1)), (2,(1,1)), (2,(0,1))\right\}$.

The cv-cycle type is actually a refinement of the cycle type, which determines the conjugacy classes of the colored permutation group. Given $\pi\in C_N\wr S_n$, we say $\pi$ has cycle type $\left\{(\lambda_1,j_1),(\lambda_2,j_2),...,(\lambda_k,j_k)\right\}$ where
\[j_i=\sum_{m=1}^{N-1}m\beta^i_m \mod N.\]
So for example the cycle type of the colored permutation above is $\left\{(3,2),(2,0),(2,0),(2,2)\right\}$. However we will only be using the cv-cycle type in this paper.

Next we want to define a subset of $\left[n-1\right]$ denoted by $\text{DEX}(\pi)$, which is a natural generalization of the set DEX appearing in \cite{sw}. We will use this set in our definition of colored Eulerian quasisymmetric functions. First, we construct a new ordered alphabet 
\[\mathcal{A}:=\left\{\widetilde{1^0}<\widetilde{2^0}<...<\widetilde{n^0}\right\}<\mathcal{E},\] 
where $\mathcal{E}$ has the same order as above, but now the letters with a tilde are less than the letters in $\mathcal{E}$.

Given any colored permutation $\pi\in C_N\wr S_n$, construct a word $\widetilde{\pi}$ of length $n$ over $\mathcal{A}$ as follows: if $i\in\text{EXC}(\pi)$, then replace $\pi_i$ by $\widetilde{\pi_i}$, otherwise leave $\pi_i$ alone. For example if $\pi=2^0,3^2,1^0,6^0,5^0,4^3\in C_4\wr S_6$, then $\text{EXC}(\pi)=\left\{1,4\right\}$ and $\widetilde{\pi}=\widetilde{2^0},3^2,1^0,\widetilde{6^0},5^0,4^3$. Then we define the set
\[\text{DEX}(\pi):=\text{DES}(\widetilde{\pi}),\]
Using the example above we have \[\text{DEX}(2^0,3^2,1^0,6^0,5^0,4^3)=\text{DES}(\widetilde{2^0},3^2,1^0,\widetilde{6^0},5^0,4^3)=\left\{3,5\right\}.\]

\vspace{.25cm}
Let $T\subseteq\left[n-1\right]:=\left\{1,2,...,n-1\right\}$ and recall that the fundamental quasisymmetric function of degree n is given by (see \cite{stanley2})
\[F_{T,n}(\textbf{x}):=\sum_{\substack{i_1\geq i_2\geq...\geq i_n\geq 1\\
i_j>i_{j+1}\text{ if }j\in T}}x_{i_1}x_{i_2}\cdots x_{i_n}.\]

\begin{defn}\label{defn 2.1}

Let $N$ be arbitrary but fixed, and given $n,j\in\mathbb{N},\vec{\alpha}\in\mathbb{N}^{N},\vec{\beta}\in\mathbb{N}^{N-1}$, we define
\[W_{n,j,\vec{\alpha},\vec{\beta}}:=\left\{\pi\in C_N\wr S_n:\exc(\pi)=j,\vec{\fix}(\pi)=\vec{\alpha},\vec{\col}(\pi)=\vec{\beta}\right\}.\]
We then define the \textit{fixed point colored Eulerian quasisymmetric functions} as
\[Q_{n,j,\vec{\alpha},\vec{\beta}}=
Q_{n,j,\vec{\alpha},\vec{\beta}}(\textbf{x}):=\sum_{\pi\in W_{n,j,\vec{\alpha},\vec{\beta}}}F_{\text{DEX}(\pi),n}(\textbf{x}).\] 
Given $j\in\mathbb{N}$ and a particular cv-cycle type $\check{\lambda}=\left\{(\lambda_1,\vec{\beta^1}),(\lambda_2,\vec{\beta^2}),...,(\lambda_k,\vec{\beta^k})\right\}$, we define
\[W_{\check{\lambda},j}:=\left\{\pi\in C_N\wr S_n:\check{\lambda}(\pi)=\check{\lambda},\exc(\pi)=j\right\},\]
where $\lambda_1+\lambda_2+...\lambda_k=n$.

We then define the \textit{cv-cycle type colored Eulerian quasisymmetric functions} by
\[Q_{\check{\lambda},j}=Q_{\check{\lambda},j}(\textbf{x})
:=\sum_{\pi\in W_{\check{\lambda},j}}
F_{\text{DEX}(\pi),n}(\textbf{x}).\]
\end{defn}

\vspace{.25cm}
It is convenient to define $\text{DEX}(\theta)=\emptyset$, $\text{des}(\theta)=\text{des}^*(\theta)=\text{maj}(\theta)=\text{exc}(\theta)=\vec{\text{fix}}(\theta)=\vec{\text{col}}(\theta)=0$, and $F_{\emptyset,0}=1$ where $\theta$ is the empty word. Thus $Q_{0,0,\vec{0},\vec{0}}=1$. We also note that the definitions of $Q_{n,j,\vec{\alpha},\vec{\beta}}$ and $Q_{\check{\lambda},j}$ agree with the definitions of $Q_{n,j,k}$ and $Q_{\lambda,j}$ in \cite{sw}, whenever $\vec{\beta}=\vec{0}$ or each $\vec{\beta^i}$ of $\check{\lambda}$ is the zero vector.

\vspace{.25cm}
Theorems \ref{thm 1.4} and \ref{thm 1.5} are obtained by applying certain ring homomorphisms to Theorem \ref{thm 1.1}. Let $\Lambda$ denote the stable principal specialization. That is, the ring homomorphism from the ring 
of quasisymmetric functions to the ring of formal power series in the variable $q$, defined by 
\[\Lambda(x_i)=q^{i-1}.\]
We will also need $\Lambda_l$, the principal specialization of order $l$ defined by
\[\Lambda_l(x_i)=\left\{\begin{array}{ll}
q^{i-1} & \text{ if } 1\leq i\leq l\\
0 & \text{ if } i>l\end{array}\right..\]

It is known (see \cite[Lemma 5.2]{gr}) that 
\[\Lambda(F_{T,n}(\textbf{x}))=\frac{q^{\sum_{i\in T}i}}{(q;q)_n}\]
and
\[\sum_{l\geq 0}\Lambda_l(F_{T,n}(\textbf{x}))p^l=\frac{p^{\left|T\right|+1}q^{\sum_{i\in T}i}}{(p;q)_{n+1}}.\]
Therefore, we will need the following lemma, whose proof is nearly identical to the proof of \cite[Lemma 2.2]{sw}.

\begin{lemma} \label{lemma 2.1}
For every $\pi\in C_N\wr S_n$ we have
\begin{equation}\left|\Exd(\pi)\right|=\left\{\begin{array}{lr}
\des^*(\pi)-1 & \text{ if }\pi_1\neq 1^0\\
\des^*(\pi) & \text{ if } \pi_1=1^0\end{array}\right.
\label{eq7}
\end{equation}

and
\begin{equation}\sum_{i\in\Exd(\pi)}i=\maj(\pi)-\exc(\pi).
\label{eq1}
\end{equation}
\end{lemma}

\begin{proof}
First define the following sets
\[J:=\left\{i\in \left[n-1\right]:i\notin\text{EXC}(\pi)\text{ and }i+1\in\text{EXC}(\pi)\right\},\]
\[K:=\left\{i\in \left[n-1\right]:i\in\text{EXC}(\pi)\text{ and }i+1\notin\text{EXC}(\pi)\right\}.\]
As in the proof of \cite[Lemma 2.2]{sw} we have $K\subseteq \Des(\pi)$ and
\[\Exd(\pi)=\left(\Des(\pi)\biguplus J\right)-K.\]
Let $J=\left\{j_1<...<j_t\right\}$ and $K=\left\{k_1<...<k_s\right\}$ and we consider two cases.

\vspace{.5cm}
Case 1: Suppose $1\notin\text{EXC}(\pi)$.

Since $n$ is never an excedance position, it follows that $t=s$ and $j_1<k_1<j_2<k_2<...<j_t<k_t$, thus
\[\sum_{i\in\text{DEX}(\pi)}i=\sum_{i\in\text{DES}(\pi)}i-\sum_{m=1}^{t}(k_m-j_m).\]
Since
\[\text{EXC}(\pi)=\biguplus_{m=1}^{t}\left\{j_m+1,j_m+2,...,k_m\right\},\]
it follows that
\[\text{exc}(\pi)=\sum_{m=1}^{t}(k_m-j_m)\]
and $\eqref{eq1}$ holds.

\vspace{.5cm}
Case 2: Suppose $1\in\text{EXC}(\pi)$.

This implies that $s=t+1$ and that $k_1<j_1<k_2<...<j_t<k_{t+1}$. Again using the fact that $\text{DEX}(\pi)=\left(\text{DES}(\pi)\biguplus J\right)-K$, we write
\[\sum_{i\in\text{DEX}(\pi)}i=\left(\sum_{i\in\text{DES}(\pi)}i\right)-k_1
-\sum_{m=1}^{t}(k_{m+1}-j_m).\]
Since
\[\text{EXC}(\pi)=\left\{1,2,...,k_1\right\}\biguplus_{m=1}^{t}\left\{j_m+1,j_m+2,...,k_{m+1}\right\},\]
we have
\[\text{exc}(\pi)=k_1+\sum_{m=1}^{t}(k_{m+1}-j_m)\]
and $\eqref{eq1}$ holds again.

To prove $\eqref{eq7}$, first consider the case when $\pi_1>1^0$. As noted above, this implies that $s=t+1$ thus $\left|\text{DEX}(\pi)\right|=\text{des}(\pi)-1=\text{des}^*(\pi)-1$. If $\pi_1=1^0$, then $s=t$ thus $\left|\text{DEX}(\pi)\right|=\text{des}(\pi)=\text{des}^*(\pi)$. If $\pi_1<1^0$, then $s=t$ and $\left|\text{DEX}(\pi)\right|=\text{des}(\pi)=\text{des}^*(\pi)-1$.
\end{proof}

\begin{proof}[Proof of Theorem \ref{thm 1.4}]

Using Lemma \ref{lemma 2.1} and Definition \ref{defn 2.1} we have
\begin{equation}
\Lambda\left(Q_{n,j,\vec{\alpha},\vec{\beta}}\right)
=\sum_{\pi\in W_{n,j,\vec{\alpha},\vec{\beta}}}
\frac{q^{\text{maj}(\pi)}q^{-j}}{(q;q)_n}.
\label{eq5}
\end{equation}

Since $h_n=F_{\emptyset,n}$, it follows that
\begin{equation}
\Lambda\left(H(z(1-q))\right)=\exp_q(z)=
\sum_{n\geq 0}\frac{z^n}{\left[n\right]_q!},
\label{eq6}
\end{equation}
since $\left[n\right]_q!=\frac{(q;q)_n}{(1-q)^n}$. Also, since $e_n=F_{\left[n-1\right],n}$, it follows that
\[\Lambda\left(E(z(1-q))\right)=\text{Exp}_q(z)=\sum_{n\geq 0}\frac{z^nq^{{n\choose 2}}}
{\left[n\right]_q!}.\]
If we then set $z\mapsto z(1-q)$ and $t\mapsto tq$ in Theorem \ref{thm 1.1} and apply $\Lambda$ to both sides, we obtain the desired result.
\end{proof}

Proving Theorem \ref{thm 1.5} takes more work, but is similar to the proofs of \cite[Lemma 2.4 and Corollary 1.4]{sw}.

\begin{proof}[Proof of Theorem \ref{thm 1.5}]

Again, using Lemma \ref{lemma 2.1} and Definition \ref{defn 2.1} we have
\[\sum_{l\geq 0}\Lambda_l\left(Q_{n,j,\vec{\alpha},\vec{\beta}}\right)p^l
=\frac{1}{(p;q)_{n+1}}\sum_{\pi\in W_{n,j,\vec{\alpha},\vec{\beta}}}
p^{\left|\text{DEX}(\pi)\right|+1}q^{\text{maj}(\pi)-j}\]
\[=\frac{1}{(p;q)_{n+1}}\sum_{\substack{\pi\in W_{n,j,\vec{\alpha},\vec{\beta}}\\
\pi(1)\neq 1^0\\}}
p^{\text{des}^*(\pi)}q^{\text{maj}(\pi)-j}
+\frac{p}{(p;q)_{n+1}}\sum_{\substack{\pi\in W_{n,j,\vec{\alpha},\vec{\beta}}\\
\pi(1)=1^0\\}}
p^{\text{des}^*(\pi)}q^{\text{maj}(\pi)-j}.\]

If we define the following quantities
\[X_{n,j,\vec{\alpha},\vec{\beta}}(p,q):=\sum_{l\geq0}\Lambda_l\left(Q_{n,j,\vec{\alpha},\vec{\beta}}\right)p^l,\]
\[Y_{n,j,\vec{\alpha},\vec{\beta}}(p,q):=\frac{1}{(p;q)_{n+1}}\sum_{\substack{\pi\in W_{n,j,\vec{\alpha},\vec{\beta}}\\
\pi(1)=1^0\\}}
p^{\text{des}^*(\pi)}q^{\text{maj}(\pi)-j},\]
\[a_{n,j,\vec{\alpha},\vec{\beta}}(p,q):=\sum_{\pi\in W_{n,j,\vec{\alpha},\vec{\beta}}}
p^{\text{des}^*(\pi)}q^{\text{maj}(\pi)},\]

then they are related by the following equation
\begin{equation}
\frac{a_{n,j,\vec{\alpha},\vec{\beta}}(p,q)}{q^j(p;q)_{n+1}}
=X_{n,j,\vec{\alpha},\vec{\beta}}(p,q)+(1-p)Y_{n,j,\vec{\alpha},\vec{\beta}}(p,q).
\label{eq9}
\end{equation}

Define a bijection
\[\gamma:\left\{\pi\in W_{n,j,\vec{\alpha},\vec{\beta}} : \pi(1)=1^0 \right\} \rightarrow
W_{n-1,j,\vec{\alpha^{(1)}},\vec{\beta}},\]
where $\vec{\alpha^{(1)}}:=(\alpha_0-1,\alpha_1,\alpha_2,...,\alpha_{N-1})$, by setting
\[\gamma(\pi)(i)=\left(\left|\pi(i+1)\right|-1\right)^{\epsilon_{i+1}}\]
for $1\leq i\leq n-1$.

For example, if $\pi=1^0,3^1,2^1,4^2$ in one-line notation, then $\gamma(\pi)=2^1,1^1,3^2$. It is clear that $\gamma$ is well-defined and a bijection, we would also like to know how $\gamma$ changes the starred descent number and the major index. Let $\pi$ be any colored permutation in the domain of $\gamma$. Since $\pi(1)=1^0$, $0\notin\text{DES}^*(\pi)$. If $i\geq 2$, then $i\in\text{DES}^*(\pi)$ iff $i-1\in\text{DES}^*(\gamma(\pi))$. Also, $1\in\text{DES}^*(\pi)$ iff $\pi(2)<1^0$ iff $0\in\text{DES}^*(\gamma(\pi))$. It follows that $\text{des}^*(\pi)=\text{des}^*(\gamma(\pi))$ and $\text{maj}(\pi)=\text{maj}(\gamma(\pi))+\text{des}^*(\pi)$. Thus
\[Y_{n,j,\vec{\alpha},\vec{\beta}}(p,q)=\frac{1}{(p;q)_{n+1}}\sum_{\pi\in W_{n-1,j,\vec{\alpha^{(1)}},\vec{\beta}}}
p^{\text{des}^*(\pi)}q^{\text{maj}(\pi)+\text{des}^*(\pi)-j}\]
\[=\frac{a_{n-1,j,\vec{\alpha^{(1)}},\vec{\beta}}(qp,q)}{q^j(p;q)_{n+1}}.\]

Note that $(1-p)/(p;q)_{n+1}=1/(qp;q)_n$, so that when we substitute this expression for $Y_{n,j,\vec{\alpha},\vec{\beta}}(p,q)$ back in to $\eqref{eq9}$ we get
\[\frac{a_{n,j,\vec{\alpha},\vec{\beta}}(p,q)}{q^j(p;q)_{n+1}}
=X_{n,j,\vec{\alpha},\vec{\beta}}(p,q)+\frac{a_{n-1,j,\vec{\alpha^{(1)}},\vec{\beta}}(qp,q)}{q^j(qp;q)_n}.\]

Let $\vec{\alpha^{(h)}}:=(\alpha_0-h,\alpha_1,\alpha_2,...,\alpha_{N-1})$, so that we can iterate this recurrence relation to obtain
\[\frac{a_{n,j,\vec{\alpha},\vec{\beta}}(p,q)}{q^j(p;q)_{n+1}}
=\sum_{h=0}^{\alpha_0}X_{n-h,j,\vec{\alpha^{(h)}},\vec{\beta}}(q^hp,q).\]

Recalling the definition of $X_{n,j,\vec{\alpha},\vec{\beta}}(p,q)$, we have
\[a_{n,j,\vec{\alpha},\vec{\beta}}(p,q)=q^j(p;q)_{n+1}\sum_{h=0}^{\alpha_0}X_{n-h,j,\vec{\alpha^{(h)}},\vec{\beta}}(q^hp,q)\]
\[=(p;q)_{n+1}\sum_{l\geq 0}p^l\sum_{h=0}^{\alpha_0}
\Lambda_l\left(Q_{n-h,j,\vec{\alpha^{(h)}},\vec{\beta}}\right)q^{hl+j}.\]

Lastly, we need the fact that $\Lambda_l(H(z))=1/(z;q)_l$ and $\Lambda_l(E(z))=(-z;q)_l$ (see \cite{stanley2}) to complete the proof,
\[\sum_{\substack{n\geq 0 \\ \pi\in C_N\wr S_n}}\frac{z^n}{(p;q)_{n+1}}
t^{\text{exc}(\pi)}r^{\vec{\text{fix}}(\pi)}s^{\vec{\text{col}}(\pi)}
q^{\text{maj}(\pi)}p^{\text{des}^*(\pi)}\]
\[=\sum_{\substack{n,j\geq 0\\ \vec{\alpha}\in\mathbb{N}^N \\ \vec{\beta}\in\mathbb{N}^{N-1}}}\frac{z^n}{(p;q)_{n+1}}
t^jr^{\vec{\alpha}}s^{\vec{\beta}} a_{n,j,\vec{\alpha},\vec{\beta}}(p,q)\]
\[=\sum_{\substack{n,j\geq 0\\ \vec{\alpha}\in\mathbb{N}^N \\ \vec{\beta}\in\mathbb{N}^{N-1}}}z^nt^jr^{\vec{\alpha}}s^{\vec{\beta}}\sum_{l\geq 0}p^l
\sum_{h=0}^{\alpha_0}\Lambda_l\left(Q_{n-h,j,\vec{\alpha^{(h)}},\vec{\beta}}\right)q^{hl+j}\]
\[=\sum_{l\geq 0}p^l\sum_{h\geq 0}(zr_0q^l)^h\Lambda_l\left(
\sum_{\substack{n,j\geq 0\\ \vec{\alpha}\in\mathbb{N}^N,\alpha_0\geq h \\ \vec{\beta}\in\mathbb{N}^{N-1}}}Q_{n-h,j,\vec{\alpha^{(h)}},\vec{\beta}}z^{n-h}(tq)^jr^{\vec{\alpha^{(h)}}}s^{\vec{\beta}}\right).\]
By Theorem \ref{thm 1.1} this is equal to
\[=\sum_{l\geq 0}\frac{p^l}{1-zr_0q^l}\Lambda_l\left(
\frac{H(r_0z)(1-tq)\left(\prod_{m=1}^{N-1}E(-s_mz)H(r_ms_mz)\right)}{\left(1+\sum_{m=1}^{N-1}s_m\right)H(tqz)-\left(tq+\sum_{m=1}^{N-1}s_m\right)H(z)}\right)\]
\[=\sum_{l\geq 0}\frac{p^l\left(\frac{1}{(r_0z;q)_l}\right)(1-tq)\left(\prod_{m=1}^{N-1}(s_mz;q)_l\frac{1}{(r_ms_mz;q)_l}\right)}
{(1-zr_0q^l)\left[\left(1+\sum_{m=1}^{N-1}s_m\right)\frac{1}{(tqz;q)_l}-\left(tq+\sum_{m=1}^{N-1}s_m\right)\frac{1}{(z;q)_l}\right]}\]
\[\hspace{-1.5cm}=\sum_{l\geq 0}\frac{p^l(1-tq)(z;q)_l(tqz;q)_l\left(\prod_{m=1}^{N-1}(s_mz;q)_l\right)}
{\left(\prod_{m=1}^{N-1}(r_ms_mz;q)_l\right)(r_0z;q)_{l+1}\left[\left(1+\sum_{m=1}^{N-1}s_m\right)(z;q)_l-\left(tq+\sum_{m=1}^{N-1}s_m\right)(tqz;q)_l\right]}.\]
\end{proof}

\vspace{.25cm}
\section{Colored necklaces and colored ornaments}

As mentioned in the introduction, Sections 3 and 4 will be devoted to the proof of Theorem \ref{thm 1.1}. In this section we introduce colored necklaces and colored ornaments. They are a multicolored generalization of the bicolored necklaces and bicolored ornaments appearing in \cite{sw}, which are in turn generalizations of the monochromatic necklaces and ornaments in \cite{gr}. We construct a bijection which shows that the cv-cycle type colored Eulerian quasisymmetric functions can be expressed as weights of colored ornaments. This bijection and its proof are similar to the bicolored versions in \cite{sw}. We will conclude this section by showing that Theorem \ref{thm 1.1} is equivalent to a certain recurrence relation.

Let $s=(s_1\geq s_2\geq ...\geq s_n)$ be a weakly decreasing sequence of positive integers. Given $\pi\in C_N\wr S_n$, we say that $s$ is $\text{DEX}(\pi)$\textit{-compatible} if $i\in\text{DEX}(\pi)$ implies that $s_i>s_{i+1}$. Then define the set Com$(\check{\lambda},j)$ as follows
\[\hspace{-.2cm}\text{Com}(\check{\lambda},j):=\left\{(\pi,s): \check{\lambda}(\pi)=\check{\lambda},\exc(\pi)=j, \text{ and }s\text{ is } \text{DEX}(\pi)\text{-compatible}\right\}.\]
Define the \textit{weight} of the pair $(\pi,s)$, denoted wt$((\pi,s))$, to be the monomial
\[\text{wt}((\pi,s)):=x_{s_1}x_{s_2}\cdots x_{s_n}.\]
Thus we can express the colored Eulerian quasisymmetric functions as follows
\[Q_{\check{\lambda},j}=\sum_{(\pi,s)\in\text{Com}(\check{\lambda},j)}\text{wt}((\pi,s)).\]

Let $\mathcal{B}$ be an infinite totally ordered alphabet with letters and order given by
\[\mathcal{B}:=\left\{1^0<1^1<...<1^{N-1}<\bar{1^0}
<2^0<2^1<...<2^{N-1}<\bar{2^0}<...\right\}.\]


Let $u$ be any positive integer, we call $\bar{u}$ a \textit{barred} letter, while letters without a bar are called \textit{unbarred}. For $0\leq m \leq N-1$ we say a letter is $m$\textit{-colored} if it is of the form $u^m$, we also say that $\bar{u^0}$ is $0$-colored. Note that only $0$-colored letters may be barred. The \textit{absolute value} of a letter is the positive integer obtained be removing any colors or bars, so $\left|u^m\right|=\left|\bar{u^0}\right|=v$.

Next we review the notion of a circular primitive word. The cyclic group of order $n$ acts on the set of words of length $n$ by cyclic rotation. So if $z$ is a generator of this cyclic group and $v=v_1,v_2,...,v_n$, then $z\cdot v=v_2,v_3,...,v_n,v_1$. A \textit{circular word}, denoted $(v)$, is the orbit of $v$ under this action. A circular word $(v)$ is called \textit{primitive} if the size of the orbit is equal to the length of the word $v$. Equivalently, a word is not primitive if it is a proper power of another word. For example the circular word $(\bar{3^0},3^0,3^2)$ is primitive, while $(4^2,3^1,4^2,3^1)=(4^2,3^1)^2$ is not primitive. One can visualize $(v)$ as a circular arrangement of letters, called a necklace, obtained from $v$ by attaching the first and last letters together. For each position of this necklace one can read the letters in a clockwise direction to obtain an element from the orbit of the circular action (see \cite{gr},\cite{reiner},\cite{sw}).

\begin{defn} \label{defn 3.1}
A \textit{colored necklace} is a circular primitive word $v$ over the alphabet $\mathcal{B}$, such that all the following rules holds

1. Every barred letter is followed by a letter of lesser or equal absolute value.

2. Every $0$-colored unbarred letter is followed by a letter of greater or equal absolute value.

3. Words of length one may not consist of a single barred letter.

\vspace{.25cm}
Note that letters with color greater than zero may be followed by any letter from $\mathcal{B}$. Also note that if $(v)$ is a colored necklace, we define its color vector in the same way we did for colored permutations. That is,
\[\vec{\text{col}}((v))=\vec{\beta}\in\mathbb{N}^{N-1}\]
means that $v$ has exactly $\beta_i$ letters with color $i\in[N-1]$.

If $v=v_1,v_2,...,v_n$, then we define the \textit{weight} of the colored necklace $(v)$, denoted wt$((v))$, to be the monomial
\[\text{wt}((v)):=x_{\left|v_1\right|}x_{\left|v_2\right|}\cdots x_{\left|v_n\right|}.\]

A \textit{colored ornament} is a multiset of colored necklaces. Formally, a colored ornament $R$ is a map with finite support from the set $\eta$ of colored necklaces to $\mathbb{N}$. We define the \textit{weight} of a colored ornament $R$, denoted wt$(R)$, to be
\[\text{wt}(R):=\prod_{(v)\in\eta}\text{wt}((v))^{R((v))}.\]

Similar to the cv-cycle type of a colored permutation, the \textit{cv-cycle type }$\check{\lambda}(R)$ of a colored ornament R is the multiset 
\[\check{\lambda}(R)=\left\{(\lambda_1,\vec{\beta^1}),(\lambda_2,\vec{\beta^2}),
...,(\lambda_k,\vec{\beta^k})\right\}\]
where each colored necklace of $R$ corresponds to precisely one pair $(\lambda_i,\vec{\beta^i})$ where this colored necklace has length $\lambda_i$ and color vector $\vec{\beta^i}$. 
\end{defn}

For example let $N=4$ and
\[R=(\bar{5^0},\bar{5^0},5^2,3^0,3^0,6^1,\bar{7^0}),(3^3,3^1),(3^3,3^1),(4^0,\bar{5^0}),(2^0),(1^3).\]
Then
\[\hspace{-.8cm}\check{\lambda}(R)=\left\{(7,(1,1,0)), (2,(1,0,1)), (2,(1,0,1)), (2,(0,0,0),
(1,(0,0,0)), (1,(0,0,1))\right\}.\]

Let $\mathcal{R}(\check{\lambda},j)$ denote the set of all colored ornaments of cv-cycle type $\check{\lambda}$, and exactly $j$ barred letters. Also, we will often refer to colored necklaces as necklaces, and colored ornaments as ornaments.

\begin{thm} \label{thm 3.2}
There exists a weight preserving bijection $f:\Com(\check{\lambda},j)\rightarrow\mathcal{R}(\check{\lambda},j)$.
\end{thm}

\begin{proof}
Let $(\pi,s)\in\text{Com}(\check{\lambda},j)$ where $s=(s_1,s_2,...,s_n)$. First we map $(\pi,s)$ to the pair $(\sigma,\alpha)$ where $\sigma\in S_n$ and $\alpha$ is a weakly decreasing sequence of $n$ letters from $\mathcal{B}$. We let $\sigma=\left|\pi\right|$, and we obtain $\alpha$ from $s$ by replacing each $s_i$ with one of the following
\[s_i\mapsto \overline{s_i^0} \hspace{.5cm}\text{ if }i\in\text{EXC}(\pi),\]
\[\hspace{2.25cm}s_i\mapsto s_i^m \hspace{.5cm}\text{ if }\epsilon_i=m\text{ and }i\notin\text{EXC}(\pi).\]
Then for each cycle $(i_1,...,i_k)$ appearing in $\sigma$, add the necklace $(\alpha_{i_1},...,\alpha_{i_k})$ to the multiset $f((\pi,s))$.

When doing an example, it helps to write the identity permutation as word on top, below that the word for the colored permutation $\pi$, and below that the sequence $s$, as follows
\[\begin{array}{cc} \text{Id} & = \\ \pi & = \\ \widetilde{\pi} & = \\ s & = \end{array}
\begin{array}{llllllll}
1 & 2 & 3 & 4 & 5 & 6 & 7 & 8 \\
8^1 & 3^0 & 2^2 & 5^0 & 1^2 & 6^1 & 4^0 & 7^0 \\
8^1 & \widetilde{3^0} & 2^2 & \widetilde{5^0} & 1^2 & 6^1 & 4^0 & 7^0 \\
6 & 5 & 5 & 4 & 4 & 4 & 4 & 3\end{array}.\]
One can check that $\text{DEX}(\pi)=\left\{1,3\right\}$ so that $s$ is $\text{DEX}(\pi)$-compatible (note that $s$ has an optional decrease from $s_7$ to $s_8$). Then
\[\sigma=(1, 8, 7, 4, 5)(2, 3)(6),\]
\[\alpha= 6^1, \bar{5^0}, 5^2, \bar{4^0}, 4^2, 4^1, 4^0, 3^0,\]
and
\[f((\pi,s))=(6^1, 3^0, 4^0, \bar{4^0}, 4^2), (\bar{5^0}, 5^2), (4^1).\]

It is clear that $f$ preserves cv-cycle type, weight, and the number of excedances of $\pi$ is equal to the number of barred letters in $f((\pi,s))$. Since $f$ preserves cv-cycle type, and since fixed points of any color cannot be excendances, it is also clear that the necklaces in $f((\pi,s))$ obey rule 3 in Definition \ref{defn 3.1}. To prove that rules 1 and 2 are also obeyed, we first prove the following 

\underline{Claim:}
$\alpha$ is a weakly decreasing sequence with respect to the order on $\mathcal{B}$.

Indeed, since $\left|\alpha_i\right|=s_i$, we know that $\left|\alpha_i\right|=s_i \geq s_{i+1}=\left|\alpha_{i+1}\right|$. So suppose $s_i=s_{i+1}$ and $\alpha_i=s_i^{m_1}$ while $\alpha_{i+1}=\overline{s_{i+1}^0}$. This means that $i\notin\text{EXC}(\pi)$ while $i+1\in\text{EXC}(\pi)$. Thus $i\in\text{DEX}(\pi)$ but $s_i=s_{i+1}$, contradicting that $s$ is $\text{DEX}(\pi)$-compatible. Also, if $\alpha_{i+1}=s_{i+1}^{m_2}$ with $m_1<m_2$, then $i$ is again an element of $\text{DEX}(\pi)$. This proves the claim.

To check that rule 1 is obeyed, suppose $\alpha_i$ is a barred letter. Then if $\sigma(i)=j$, we must have $i<j$. By the claim above, $\alpha_i\geq\alpha_j$. To check rule 2, suppose $\alpha_i$ is $0$-colored and unbarred. Then $\sigma(i)=j$ with $i>j$ and the claim tells us that $\alpha_i\leq\alpha_j$.

To show that $f$ is well-defined, it remains to show that each word in $f((\pi,s))$ is primitive. Suppose $(\alpha_{i_1},\alpha_{i_2},...,\alpha_{i_k})$ is a nonprimitive necklace in $f((\pi,s))$ obtained from the cycle $(i_1,i_2,...,i_k)$ of $\sigma$, where $i_1$ is the smallest element of the cycle. Thus for some divisor $d$ of $k$ we have $(\alpha_{i_1},\alpha_{i_2},...,\alpha_{i_k})=(\alpha_{i_1},\alpha_{i_2},...,\alpha_{i_d})^{k/d}$. In particular we have
\begin{equation} \alpha_{i_1}=\alpha_{i_{d+1}} \text{ and } i_1<i_{d+1}.
\label{eq2}
\end{equation}
Since the sequence $\alpha$ is weakly decreasing, this implies that $\alpha_i=\alpha_{i_1}$ for all $i\in B:=\left\{i:i_1\leq i \leq i_{d+1}\right\}$, and $B\cap\text{DEX}(\pi)=\emptyset$. Moreover, either $B\cap\text{EXC}(\pi)=B$ or $\emptyset$, and $\epsilon_i=\epsilon_{i_1}$ for all $i\in B$. So in fact $B\cap\text{DES}(\pi)=\emptyset$ and
\begin{equation} \sigma(i_1)<\sigma(i_1+1)<\sigma(i_1+2)<...<\sigma(i_{d+1}).
\label{eq3}
\end{equation}

From $\eqref{eq3}$, we find that $i_2=\sigma(i_1)<\sigma(i_{d+1})=i_{d+2}$. Since $(\alpha_{i_1},\alpha_{i_2},...,\alpha_{i_k})=(\alpha_{i_1},\alpha_{i_2},...,\alpha_{i_d})^{k/d}$, we now have $\alpha_{i_2}=\alpha_{i_{d+2}}$ with $i_2<i_{d+2}$, similar to $\eqref{eq2}$. The same argument will show $i_3=\sigma(i_2)<\sigma(i_{d+2})=i_{d+3}$, and we can repeat this argument until $i_{k-d+1}=\sigma(i_{k-d})<\sigma(i_k)=i_1$, contradicting the minimality of $i_1$.

Thus far we have proved that $f:\text{Com}(\check{\lambda},j)\rightarrow\mathcal{R}(\check{\lambda},j)$ is well-defined. Next, we will describe the inverse map $g:\mathcal{R}(\check{\lambda},j)\rightarrow\text{Com}(\check{\lambda},j)$ and show that it is well-defined. Let $R\in\mathcal{R}(\check{\lambda},j)$ and if $R$ has any repeated necklaces, fix some total order on these repeated necklaces. For each position $x$ of each necklace, let $w_x$ denote the infinite word obtained by reading the necklace clockwise starting at position $x$. Let $w_x>w_y$ mean that $w_x$ is lexicographically larger than $w_y$, using the order on $\mathcal{B}$. If $w_x=w_y$ for distinct positions $x,y$, then it must be that $x,y$ are positions in distinct copies of a repeated necklace, since words are primitive. We can then break the tie using the total order on repeated necklaces.

This totally orders all the positions on all of the necklaces of $R$ by letting $x>y$ iff 

\vspace{.25cm}
(1) $w_x>w_y$

or
 
(2) $w_x=w_y$ and $x$ is in a necklace which is larger in the total order on these repeated necklaces.
\vspace{.25cm}

If $x$ is the $i^{\text{th}}$ largest position of $R$, then we replace the letter in position $x$ by $i$. After doing this for each position, we have a permutation denoted $\sigma(R)\in S_n$ written in cycle form. We then obtain a colored permutation denoted $\pi(R)$ by setting $\pi(R)(i)=(\sigma(R)(i))^{\epsilon_x}$ where $\epsilon_x$ is the color of the letter formerly occupying position $x$. A sequence $s(R)$ is obtained by simply taking the weakly decreasing rearrangement of the absolute values of all the letters appearing in $R$. We then set $g(R)=(\pi(R),s(R))$.

For example, consider the following ornament
\[\begin{array}{cccccccccccccc}
R=&(5^1,&\bar{5^0},&3^0)&<&(5^1,&\bar{5^0},&3^0),&(4^1,&3^0,&3^1,&4^1),&(3^1),&(3^0)
\end{array}.\]
By ranking each position, we obtain $\sigma(R)$ as follows
\[\begin{array}{rlllllllllllll}
R=&(5^1,&\bar{5^0},&3^0)&<&(5^1,&\bar{5^0},&3^0),&(4^1,&3^0,&3^1,&4^1),&(3^1),&(3^0) \\
\sigma(R) =&(4,&2,&10)& &(3,&1,&9)&(6,&11,&7,&5)&(8)&(12)
\end{array}.\]
So $g(R)$ is the pair
\[\begin{array}{rllllllllllll}
\text{Id}= & 1 & 2 & 3 & 4 & 5 & 6 & 7 & 8 & 9 & 10 & 11 & 12 \\
\pi(R)= & 9^0 & 10^0 & 1^1 & 2^1 & 6^1 & 11^1 & 5^1 & 8^1 & 3^0 & 4^0 & 7^0 & 12^0 \\
s(R)= & 5 & 5 & 5 & 5 & 4 & 4 & 3 & 3 & 3 & 3 & 3 & 3
\end{array}.\]

It is easy to see that $\check{\lambda}(\pi(R))=\check{\lambda}$, and that $g$ does not depend on the ordering of repeated necklaces in $R$. Also, it follows from rules 1 and 2 in Definition \ref{defn 3.1} that if $x$ is the $i^{\text{th}}$ largest position, then the letter in position $x$ is barred iff $i\in\text{EXC}(\pi(R))$, thus $\text{exc}(\pi(R))=j$.

To show that $g$ is well-defined, it remains to show that the sequence $s$ is $\text{DEX}(\pi)$-compatible. Suppose $s_i=s_{i+1}$. Let $x$ be the $i^{\text{th}}$ largest position in $R$, and $y$ be the $(i+1)^{\text{th}}$ largest position, in particular $x>y$. Given any word $w$, let $F(w)$ denote the first letter of the word. So $s_i=s_{i+1}$ means that $\left|F(w_x)\right|=\left|F(w_y)\right|$. If $F(w_x)>F(w_y)$, then one can easily check that $i\notin\text{DEX}(\pi)$ as desired.

So assume $F(w_x)=F(w_y)$. Let $u$ denote the position immediately following $x$ cyclicly, and let $v$ denote the position immediately following $y$. Since $x>y$ it follows that $u>v$. Since $\sigma(R)(i)$ is equal to the rank of position $u$, and $\sigma(R)(i+1)$ is is equal to the rank of position $v$, we have $\sigma(R)(i)<\sigma(R)(i+1)$. Since $F(w_x)=F(w_y)$, then $\epsilon_x=\epsilon_y$ and this implies that $i\notin\Des(\pi(R))$. Moreover, since either $i,i+1\in\text{EXC}(\pi(R))$ or $i,i+1\notin\text{EXC}(\pi(R))$ we have $i\notin\text{DEX}(\pi(R))$. Thus the map $g$ is well-defined.

The proof of Theorem \ref{thm 3.2} will be complete once we show that $f\circ g=g\circ f=\text{id}$. It not hard to see that $f\circ g=\text{id}$, and that if we apply $g\circ f$ to $(\pi,s)$ we will recover the sequence $s$. So we need to prove that applying $g\circ f$ to $(\pi,s)$ will also bring us back to the colored permutation $\pi$.

Let $(\pi,s)\mapsto(\sigma,\alpha)$ in the first step of $f$, and let $p_i$ be the position occupied by $\alpha_i$ in $f((\pi,s))$. Order the cycles of $\sigma$ from largest to smallest so that the minimum elements of the cycles increase. Use this to order repeated necklaces in $f((\pi,s))$ so that we know how to break ties if $w_{p_i}=w_{p_j}$. We want to show the following:

\vspace{.5cm}
(i) if $i<j$, then $w_{p_i}\geq w_{p_j}$,

\vspace{.5cm}
and

\vspace{.5cm}
(ii) if $i<j$ and $w_{p_i}=w_{p_j}$, then $i$ is in a cycle of $\sigma$ whose minimum element is less than the minimum element of the cycle containing $j$.

\vspace{.5cm}
In order to prove both (i) and (ii), we first establish that

\vspace{.5cm}
(iii) If $i<j$ and $w_{p_i}\leq w_{p_j}$, then $\alpha_i=\alpha_j$ and $\sigma(i)<\sigma(j)$.

\vspace{.5cm}
Indeed, $\alpha$ is weakly decreasing so that $i<j$ implies $\alpha_i\geq\alpha_j$. And $w_{p_i}\leq w_{p_j}$ implies that $\alpha_i=F(w_{p_i})\leq F(w_{p_j})=\alpha_j$, so $\alpha_i=\alpha_j$. This implies that $\alpha_i=\alpha_{i+1}=...=\alpha_j$, which means $s_i=s_{i+1}=...=s_j$, $\epsilon_i=\epsilon_{i+1}=...=\epsilon_j$, and all the letters  $\alpha_i,\alpha_{i+1},...,\alpha_j$ are either all barred or all unbarred. Since $s$ is $\Exd(\pi)$-compatible, $k\notin\text{DEX}(\pi)$ for $i\leq k\leq j-1$. This implies that $\left|\pi(i)\right|<\left|\pi(i+1)\right|<...<\left|\pi(j)\right|$, which means that $\sigma(i)<\sigma(i+1)<...<\sigma(j)$. This establishes (iii).

To prove (i), suppose $i<j$ but $w_{p_i}<w_{p_j}$. Using (iii), we have $\sigma(i)<\sigma(j)$ and $\alpha_i=\alpha_j$. Since $F(w_{p_i})=F(w_{p_j})$, we must have $w_{p_{\sigma(i)}}<w_{p_{\sigma(j)}}$. Now apply (iii) again with $\sigma(i),\sigma(j)$ taking the role of $i,j$. Then $\sigma^2(i)<\sigma^2(j)$ and $\alpha_{\sigma(i)}=\alpha_{\sigma(j)}$, which implies $w_{p_{\sigma^2(i)}}<w_{p_{\sigma^2(j)}}$. Apply (iii) again to obtain $\sigma^3(i)<\sigma^3(j)$, $\alpha_{\sigma^2(i)}=\alpha_{\sigma^2(j)}$ and $w_{p_{\sigma^3(i)}}<w_{p_{\sigma^3(j)}}$. By repeating this argument, we see that $\alpha_{\sigma^m(i)}=\alpha_{\sigma^m(j)}$ for all $m$, but this implies that $w_{p_i}=w_{p_j}$, a contradiction.

To prove (ii), suppose $i<j$ and $w_{p_i}=w_{p_j}$. Using (iii) we have $\sigma(i)<\sigma(j)$, and $w_{p_i}=w_{p_j}$ implies $w_{p_{\sigma(i)}}=w_{p_{\sigma(j)}}$. Applying (iii) again we have $\sigma^2(i)<\sigma^2(j)$ and $w_{p_{\sigma^2(i)}}=w_{p_{\sigma^2(j)}}$. Repeating this argument, we have $\sigma^m(i)<\sigma^m(j)$ for all $m$. Thus the cycle of $\sigma$ containing $i$ has a smaller minimum element, than the cycle containing $j$.

This completes the proof that $f:\text{Com}(\check{\lambda},j)\rightarrow\mathcal{R}(\check{\lambda},j)$ is a bijection.
\end{proof}

Previously, we had expressed the cv-cycle type colored Eulerian quasisymmetric functions as a sum of weights of pairs $(\pi,s)$. Using Theorem \ref{thm 3.2} we can now express is it as a sum of weights of ornaments. 

\begin{cor} \label{cor 3.3}
\[Q_{\check{\lambda},j}=\sum_{(\pi,s)\in\Com(\check{\lambda},j)}\wt((\pi,s))
=\sum_{R\in\mathcal{R}(\check{\lambda},j)}\wt(R).\]
\end{cor}

\begin{remark} \label{remark1}
It is possible to use Corollary \ref{cor 3.3} to prove that $Q_{\check{\lambda},j}$ is also a symmetric function. One method is to use the ornament description of $Q_{\check{\lambda},j}$ to derive a colored analog of \cite[Corollary 6.1]{sw}, which involves plethysm (see \cite{wachs}). Another possible method is a bijective approach as in \cite[Theorem 5.8]{sw}. We plan to presents the details of both proofs in a forthcoming paper.
\end{remark}

From Corollary \ref{cor 3.3}, we obtain the following results concerning the fixed point colored Eulerian quasisymmetric functions.

\begin{cor} \label{cor 3.4}
\[Q_{n,j,\vec{\alpha},\vec{\beta}}=Q_{n-\left|\vec{\alpha}\right|,j,\vec{0},\vec{\beta}-(\alpha_1,\alpha_2,...,\alpha_{N-1})}\prod_{k=0}^{N-1}h_{\alpha_k},\]
where $\left|\vec{\alpha}\right|:=\sum_{k=0}^{N-1}\alpha_k$, and recall that $h_{\alpha_k}$ is the complete homogeneous symmetric function of degree $\alpha_k$.
\end{cor}

\begin{cor} \label{cor 3.5}
Theorem \ref{thm 1.1} is equivalent to
\begin{equation}
\sum_{\substack{n,j\geq 0 \\ \vec{\beta}\in\mathbb{N}^{N-1}}} Q_{n,j,\vec{0},\vec{\beta}}z^nt^js^{\vec{\beta}}=\frac{(1-t)\prod_{m=1}^{N-1}E(-s_mz)}{\left(1+\sum_{m=1}^{N-1}s_m\right)H(tz)-\left(t+\sum_{m=1}^{N-1}s_m\right)H(z)}.
\label{eq cor 3.5}
\end{equation}
\end{cor}

\begin{proof}
In one direction, take the formula from Theorem \ref{thm 1.1} and simply set $r_0=r_1=...=r_{N-1}=0$. For the other direction, start with \eqref{eq cor 3.5} and multiply both sides by $H(r_0z)\prod_{m=1}^{N-1}H(r_ms_mz)$. The left hand side becomes
\[\left(\sum_{n\geq 0}r_0^nz^nh_n\right)
\left(\prod_{m=1}^{N-1}\sum_{n\geq 0}(r_ms_m)^nz^nh_n\right)\left(\sum_{\substack{n,j\geq 0 \\ \vec{\beta}\in\mathbb{N}^{N-1}}}Q_{n,j,\vec{0},\vec{\beta}}z^nt^js^{\vec{\beta}}\right)\]
\[=\sum_{n\geq 0}z^n \sum_{\substack{j\geq 0 \\
\vec{\beta}\in\mathbb{N}^{N-1} \\
\vec{\alpha}\in\mathbb{N}^N }}
Q_{n-\left|\vec{\alpha}\right|,j,\vec{0},\vec{\beta}}t^js^{\vec{\beta}+(\alpha_1,...,\alpha_{N-1})}r^{\vec{\alpha}}\prod_{k=0}^{N-1}h_{\alpha_k}.\]
By Corollary \ref{cor 3.4}, this is equal to the left hand side of Theorem \ref{thm 1.1}.
\end{proof}

\begin{cor} \label{cor 3.6}
\eqref{eq cor 3.5} is equivalent to the recurrence relation
\[Q_{n,j,\vec{0},\vec{\beta}}=
\sum_{\substack{0\leq i\leq n-2 \\ j-n+i<k<j}}Q_{i,k,\vec{0},\vec{\beta}}h_{n-i}\]
\[+\sum_{m=1}^{N-1}\chi\left(\beta_m>0\right)\left(\sum_{\substack{0\leq i\leq n-1 \\ j-n+i<k\leq j}}Q_{i,k,\vec{0},\vec{\beta}(\hat{m})}h_{n-i}\right)\]
\[+\chi\left(j=0\right)\chi\left(\left|\vec{\beta}\right|=n\right)(-1)^n
\prod_{m=1}^{N-1}e_{\beta_m},\]

where if $\vec{\beta}=(\beta_1,\beta_2,...,\beta_{N-1})$ then \[\vec{\beta}(\hat{m}):=(\beta_1,...,\beta_{m-1},\beta_m-1,\beta_{m+1},...,\beta_{N-1}),\]

and $\chi(P)=0$ if the statement $P$ is false, and $\chi(P)=1$ if the statement $P$ is true.
\end{cor}

\begin{proof}
Let
\[I=\sum_{\substack{n,j\geq 0 \\ \vec{\beta}\in\mathbb{N}^{N-1}}} Q_{n,j,\vec{0},\vec{\beta}}z^nt^js^{\vec{\beta}}.\]

Then the recurrence relation is equivalent to
\[I=\sum_{\substack{n,j\geq 0 \\ \vec{\beta}\in\mathbb{N}^{N-1}}} 
z^n\sum_{\substack{0\leq i\leq n-2 \\ j-n+i<k<j}}
Q_{i,k,\vec{0},\vec{\beta}}h_{n-i}t^js^{\vec{\beta}}\]
\[+\sum_{m=1}^{N-1}\sum_{\substack{n,j\geq 0 \\ \vec{\beta}\in\mathbb{N}^{N-1}}} 
z^n\sum_{\substack{0\leq i\leq n-1 \\ j-n+i<k\leq j}}
Q_{i,k,\vec{0},\vec{\beta}(\hat{m})}h_{n-i}t^js^{\vec{\beta}(\hat{m})}s_m
+\sum_{\vec{\beta}\in\mathbb{N}^{N-1}}
(-z)^{\left|\vec{\beta}\right|}s^{\vec{\beta}}\prod_{m=1}^{N-1}e_{\beta_m}\]

\[=\sum_{\substack{n,k\geq 0 \\ \vec{\beta}\in\mathbb{N}^{N-1}}} 
z^n\sum_{0\leq i\leq n-2}
Q_{i,k,\vec{0},\vec{\beta}}s^{\vec{\beta}}h_{n-i}
\sum_{j=k+1}^{k+n-i-1}t^j\]
\[+\sum_{m=1}^{N-1}\sum_{\substack{n,k\geq 0 \\ \vec{\beta}\in\mathbb{N}^{N-1}}} 
z^n\sum_{0\leq i\leq n-1}
Q_{i,k,\vec{0},\vec{\beta}(\hat{m})}s^{\vec{\beta}(\hat{m})}
h_{n-i}s_m\sum_{j=k}^{k+n-i-1}t^j+\prod_{m=1}^{N-1}E(-s_mz)\]

\[=\sum_{\substack{n,k\geq 0 \\ \vec{\beta}\in\mathbb{N}^{N-1}}} 
z^n\sum_{0\leq i\leq n-2}
Q_{i,k,\vec{0},\vec{\beta}}t^ks^{\vec{\beta}}th_{n-i}
\left[n-i-1\right]_t\]
\[+\sum_{m=1}^{N-1}\sum_{\substack{n,k\geq 0 \\ \vec{\beta}\in\mathbb{N}^{N-1}}} 
z^n\sum_{0\leq i\leq n-1}
Q_{i,k,\vec{0},\vec{\beta}(\hat{m})}t^ks^{\vec{\beta}(\hat{m})}
s_mh_{n-i}\left[n-i\right]_t+\prod_{m=1}^{N-1}E(-s_mz)\]

\[=I\sum_{n\geq 2}t\left[n-1\right]_th_nz^n
+\sum_{m=1}^{N-1}I\sum_{n\geq 1}s_m\left[n\right]_th_nz^n
+\prod_{m=1}^{N-1}E(-s_mz).\]

If we let
\[A:=\sum_{n\geq 2}t\left[n-1\right]_th_nz^n\]
and
\[B:=\sum_{m=1}^{N-1}\sum_{n\geq 1}s_m\left[n\right]_th_nz^n,\]
then
\begin{equation}
I=\frac{\prod_{m=1}^{N-1}E(-s_mz)}{1-A-B}.
\label{eq4}
\end{equation}

Next we compute the denominator of this expression,
\[1-A-B=1+\sum_{n\geq 2}h_nz^nt\left(\frac{t^{n-1}-1}{1-t}\right)
+\sum_{m=1}^{N-1}\sum_{n\geq 1} h_nz^ns_m\left(\frac{t^n-1}{1-t}\right)\]
\[=\frac{1}{1-t}\left[1-t+H(tz)-tz-1-t(H(z)-z-1)
+\sum_{m=1}^{N-1}(H(tz)-H(z))s_m\right]\]
\[=\frac{1}{1-t}\left[\left(1+\sum_{m=1}^{N-1}s_m\right)H(tz)
-\left(t+\sum_{m=1}^{N-1}s_m\right)H(z)\right].\]

Substituting this back into $\eqref{eq4}$ gives the desired result.
\end{proof}

\section{Colored banners}

The previous section has shown that Theorem \ref{thm 1.1} is equivalent to the recurrence relation appearing in Corollary \ref{cor 3.6}. This section will be devoted to establishing this recurrence relation, thus proving Theorem \ref{thm 1.1}. There are two cases which will be treated separately, the case $\left|\vec{\beta}\right|=n$ and the case $\left|\vec{\beta}\right|<n$ (recall that the absolute value of a vector $\vec{\beta}\in\mathbb{N}^{N-1}$ is $\left|\vec{\beta}\right|=\sum_{m=1}^{N-1}\beta_m$).

First we consider the case $\left|\vec{\beta}\right|=n$, and define
\[D_{n,\vec{\beta}}(\xx):=Q_{n,0,\vec{0},\vec{\beta}}(\xx).\]
(Note that $\xx:=\left\{x_1,x_2,...\right\}$ denotes our usual set of commuting variables which we often omit, but we include here for the sake of clarity in the proof of Theorem \ref{cor 4.2}). Our goal is to compute the following recurrence relation for $D_{n,\vec{\beta}}(\xx)$,
and one can then check that it agrees with the recurrence relation appearing in Corollary \ref{cor 3.6} in the case when $\left|\vec{\beta}\right|=n$.

\begin{thm} \label{cor 4.2}
\[D_{n,\vec{\beta}}(\xx)=(-1)^n\left(\prod_{m=1}^{N-1}e_{\beta_m}(\xx)\right)
+\sum_{m=1}^{N-1}\chi(\beta_m>0)D_{n-1,\vec{\beta}(\hat{m})}(\xx)h_1(\xx),\]
recalling that $\vec{\beta}(\hat{m})=
(\beta_1,...,\beta_{m-1},\beta_m-1,\beta_{m+1},...,\beta_{N-1})$.
\end{thm}

\begin{proof}

Similar to the definition of $\mathcal{R}(\check{\lambda},j)$, we let $\mathcal{R}(n,j,\vec{\alpha},\vec{\beta})$ denote the set of all ornaments of size $n$ with $j$ barred letters, $\beta_i$ letters of color $i\in[N-1]$, and $\alpha_i$ necklaces consisting of a single $i$-colored letter where $0\leq i\leq N-1$ (as usual, $N$ is arbitrary but fixed). Hence
\[D_{n,\vec{\beta}}(\xx)=\sum_{R\in\mathcal{R}(n,0,\vec{0},\vec{\beta})}\wt(R).\]

Since $\left|\vec{\beta}\right|=n$, the key fact is that the necklace rules of Definition \ref{defn 3.1} present no restrictions, since there are no $0$-colored letters in this case. Therefore $\mathcal{R}(n,0,\vec{0},\vec{\beta})$ can be viewed as a set of Gessel-Reutenauer ornaments as in \cite{gr}, but over the alphabet
\[\left\{1^1,1^2,...,1^{N-1},2^1,2^2,...,2^{N-1},...\right\}.\]
Given a necklace $(v)$ over this alphabet where $v=v_1^{\epsilon_1},v_2^{\epsilon_2},...,v_n^{\epsilon_n}$ with $v_i\in\mathbb{P}$ and $\epsilon_i\in[N-1]$, we define a new weight by
\[\wtt((v)):=x_{v_1,\epsilon_1}x_{v_2,\epsilon_2}\cdots x_{v_n,\epsilon_n},\]
where
\[X:=\left\{x_{1,1},x_{1,2},...,x_{1,N-1},x_{2,1},x_{2,2},...,x_{2,N-1},...\right\}\]
is a set of commuting variables.

Then by \cite[Theorem 3.6]{gr} we have
\[\sum_{\substack{\vec{\beta}\in\mathbb{N}^{N-1} \\ \left|\vec{\beta}\right|=n}}
D_{n,\vec{\beta}}(\xx)s^{\vec{\beta}}
=\sum_{\substack{\vec{\beta}\in\mathbb{N}^{N-1} \\ \left|\vec{\beta}\right|=n}}
\sum_{R\in\mathcal{R}(n,0,\vec{0},\vec{\beta})}\wtt(R)\bigg|_{x_{i,j}=x_is_j}
=D_n(X)\bigg|_{x_{i,j}=x_is_j},\]
where $D_n$ is the quasisymmetric generating function for derangements in $S_n$, as described in \cite[Section 8]{gr}. By Equation (8.2) of \cite{gr}, $D_n$ satisfies the following recurrence
\[D_n(X)=h_1(X)D_{n-1}(X)+(-1)^ne_n(X).\]

Next we compute the right hand side of this equation evaluated at $x_{i,j}=x_is_j$. First we have
\[D_{n-1}(X)\bigg|_{x_{i,j}=x_is_j}=\sum_{\substack{\vec{\beta}\in\mathbb{N}^{N-1} \\ \left|\vec{\beta}\right|=n-1}}
D_{n-1,\vec{\beta}}(\xx)s^{\vec{\beta}}.\]

Next,
\[h_1(X)\bigg|_{x_{i,j}=x_is_j}=\sum_{i\geq 1}x_i\left(\sum_{m=1}^{N-1}s_m\right)=h_1(\xx)\sum_{m=1}^{N-1}s_m.\]

And finally
\[e_n(X)\bigg|_{x_{i,j}=x_is_j}=\sum_{\substack{\vec{\beta}\in\mathbb{N}^{N-1} \\ \left|\vec{\beta}\right|=n}}
\prod_{m=1}^{N-1}s_m^{\beta_m}e_{\beta_m}(\xx)
=\sum_{\substack{\vec{\beta}\in\mathbb{N}^{N-1} \\ \left|\vec{\beta}\right|=n}}s^{\vec{\beta}}
\prod_{m=1}^{N-1}e_{\beta_m}(\xx).\]

Thus 
\[\sum_{\substack{\vec{\beta}\in\mathbb{N}^{N-1} \\ \left|\vec{\beta}\right|=n}}
D_{n,\vec{\beta}}(\xx)s^{\vec{\beta}}=D_n(X)\bigg|_{x_{i,j}=x_is_j}\]
\[=h_1(\xx)\left(\sum_{m=1}^{N-1}s_m\sum_{\substack{\vec{\beta}\in\mathbb{N}^{N-1} \\ \left|\vec{\beta}\right|=n-1}}
D_{n-1,\vec{\beta}}(\xx)s^{\vec{\beta}}\right)
+(-1)^n\sum_{\substack{\vec{\beta}\in\mathbb{N}^{N-1} \\ \left|\vec{\beta}\right|=n}}s^{\vec{\beta}}
\prod_{m=1}^{N-1}e_{\beta_m}(\xx).\]

Extracting the coefficient of $s^{\vec{\beta}}$ from both sides gives the desired result.

\end{proof}

It now remains to consider the case $\left|\vec{\beta}\right|<n$ for establishing the recurrence relation in Corollary \ref{cor 3.6}. For this we introduce colored banners, which are a generalization of the banners introduced in \cite{sw}.

\begin{defn} \label{defn 4.3}
A \textit{colored banner} (or simply banner) is a word $B$ over the alphabet $\mathcal{B}$ such that 

1. if $B(i)$ is barred then $\left|B(i)\right|\geq\left|B(i+1)\right|,$

2. if $B(i)$ is $0$-colored and unbarred, then $\left|B(i)\right|\leq \left|B(i+1)\right|$ or $i$ equals the length of $B$,

3. the last letter of $B$ is unbarred.

\end{defn}

Recall that a Lyndon word over an ordered alphabet is a word that is strictly lexicographically larger than all its circular rearrangements. And a Lyndon factorization of a word is a factorization into a lexicographically weakly increasing sequence of Lyndon words. It is a fact that every word has a unique Lyndon factorization. We say that a word of length $n$ has Lyndon type $\lambda$ (where $\lambda$ is a partition of $n$) if parts of $\lambda$ equal the lengths of the words in the Lyndon factorization (see \cite[Theorem 5.1.5]{lothaire}).

We will apply Lyndon factorization to banners, but we will do so using a new order $<_B$ on the alphabet $\mathcal{B}$ as follows
\[1^1<_B 1^2<_B ... <_B 1^{N-1}\]
\[<_B 2^1<_B 2^2<_B ... <_B 2^{N-1}\]
\[<_B 3^1<_B 3^2<_B ... <_B 3^{N-1}\]
\[\vdots\]
\[1^0<_B\bar{1^0}<_B 2^0<_B\bar{2^0}<_B 3^0<_B\bar{3^0}
<_B ...\]
(The reason for choosing this order will become apparent in the proof of Theorem \ref{thm 4.6}). We define the weight wt$(B)$ of a banner $B(i),...,B(n)$ to be the monomial $x_{\left|B(1)\right|}...x_{\left|B(n)\right|}$. And we define the cv-cycle type of a banner B to be the multiset 
\[\check{\lambda}(B)=\left\{(\lambda_1,\vec{\alpha^1}),...,(\lambda_k,\vec{\alpha^k})\right\}\]
if $B$ has Lyndon type $\lambda$ with respect to $<_B$, and the corresponding word of length $\lambda_i$ in the Lyndon factorization has color vector $\vec{\alpha^i}$. Then $K(\check{\lambda},j)$ will denote the set of all banners of cv-cycle type $\check{\lambda}$ with exactly $j$ barred letters.

\begin{thm} \label{thm 4.4}
There exists a weight preserving bijection from $\mathcal{R}(\check{\lambda},j)$ to $K(\check{\lambda},j)$, consequently
\[Q_{\check{\lambda},j}=\sum_{B\in K(\check{\lambda},j)}
\wt(B).\]
\end{thm}

\begin{proof}
The proof uses Lyndon factorization and is identical to the proof of \cite[Theorem 3.6]{sw}.
\end{proof}

\begin{defn} \label{defn 4.5}
A \textit{$0$-colored marked sequence}, denoted $(\omega,b,0)$, is a weakly increasing sequence $\omega$ of positive integers, together with a positive integer $b$, which we call the mark, such that $1\leq b< \text{length}(\omega)$. The set of all $0$-colored marked sequences with length$(\omega)=n$ and mark equal to $b$ will be denoted $M(n,b,0)$.

For $1\leq m\leq N-1$, an \textit{$m$-colored marked sequence}, denoted $(\omega,b,m)$, is a weakly increasing sequence $\omega$ of positive integers, together with a nonnegative integer $b$ such that $0\leq b< \text{length}(\omega)$. The set of all $m$-colored marked sequences with length$(\omega)=n$ and mark equal to $b$ will be denoted $M(n,b,m)$.

We will use colored marked sequences in Theorem \ref{thm 4.6} below, where one can think of the map $\gamma$ as removing a colored marked sequence $(\omega,b,m)$ from a banner. The sequence $\omega$ corresponds to the absolute values of the letters removed, $b$ corresponds to the number of barred letters removed, and one of the letters removed has color $m$ while the rest of the letters removed all have color 0.

Let $K_0(n,j,\vec{\beta})$ denote the set of all banners of length $n$, with Lyndon type having no parts of size 1, color vector equal to $\vec{\beta}$, and $j$ bars. For $m\in [N-1]$ and $\beta_m>0$, define
\[X_m:=
\biguplus_{\substack{0\leq i\leq n-1 \\
j-n+i< k\leq j}}
K_0(i,k,\vec{\beta}(\hat{m}))\times M(n-i,j-k,m),\]
and let $X_m:=0$ if $\beta_m=0$. We also define
\[X_0=:\biguplus_{\substack{0\leq i\leq n-2 \\ j-n+i<k<j}}
K_0(i,k,\vec{\beta})\times M(n-i,j-k,0).\]

\end{defn}

\begin{thm} \label{thm 4.6}
If $\left|\vec{\beta}\right|<n$, then there is a bijection
\[\gamma:K_0(n,j,\vec{\beta})\rightarrow\biguplus_{m=0}^{N-1}X_m\]
such that if $\gamma(B)=(B',(\omega,b))$, then $\text{wt}(B)=\text{wt}(B')\text{wt}(\omega)$
where the weight of any sequence of positive integers is the monomial $x_{\omega_1}x_{\omega_2}\cdots x_{\omega_{n-i}}$.
\end{thm}

\begin{cor} \label{cor 4.7}
Theorem \ref{thm 4.6} establishes the recurrence relation appearing in Corollary \ref{cor 3.6} in the case that $\left|\vec{\beta}\right|<n$.
\end{cor}

\begin{proof}
This follows from the fact that
\[\sum_{(\omega,j-k)\in M(n-i,j-k,m)}\text{wt}(\omega)
=h_{n-i}.\]
\end{proof}

In order to prove Theorem \ref{thm 4.6}, we will need the following Lemma (see \cite[Lemma 4.3]{dw}).

\begin{lemma} \label{lemma 4.8}
Let $B$ be a banner. If the Lyndon type of $B$ has no parts of size one, then $B$ has a unique increasing factorization (with respect to $<_B$). By increasing factorization of $B$, we mean that $B$ has the form $B=B_1\cdot B_2\cdots B_d$ where each $B_i$ has the form 
\[B_i=(\underbrace{a_i,...,a_i}_{p_i\text{ times}})\cdot u_i,\]
where $a_i\in\mathcal{B}$, $p_i>0$, and $u_i$ is a word of positive length over the alphabet $\mathcal{B}$ whose letters are all strictly less than $a_i$ with respect to $<_B$, and $a_1\leq_B a_2\leq_B ...\leq_B a_d$. Note that the increasing factorization is a refinement of the Lyndon factorization.
\end{lemma}

For example, the Lyndon factorization of the word \[(6^1,1^2,5^1,\bar{4^0},\bar{4^0},4^1,4^0,\bar{4^0},3^2,5^0,7^1)\]
is
\[(6^1,1^2,5^1)\cdot(\bar{4^0},\bar{4^0},4^1,4^0,\bar{4^0},3^2)\cdot(5^0,7^1),\]
which has no parts of size one, so its increasing factorization is
\[(6^1,1^2,5^1)\cdot(\bar{4^0},\bar{4^0},4^1,4^0)\cdot(\bar{4^0},3^2)\cdot(5^0,7^1).\]

\vspace{.25cm}
Next we prove Theorem \ref{thm 4.6}. In the case $N=1$, this proof reduces to the proof of \cite[Theorem 3.7]{sw}. In the general case that $N>1$ (and $\left|\vec{\beta}\right|<n$), the proof is inspired by \cite[Theorem 3.7]{sw}, but significantly more complicated. 

\begin{proof}[Proof of Theorem \ref{thm 4.6}]
Describing $\gamma$ (and its inverse) requires us to consider many different cases. For convenience we will make a note of which case $\gamma(B)$ falls under when considering $\gamma^{-1}(\gamma(B))$ (and vice versa) so that one can check that $\gamma$ is indeed a bijection.

First, we take the increasing factorization of $B$, say $B=B_1\cdot B_2\cdots B_d$. Let
\[B_d=(\underbrace{a,...,a}_{p\text{ times}})\cdot u,\]
where $a\in\mathcal{B}$, $p>0$, and $u$ is a word of positive length over the alphabet $\mathcal{B}$ whose letters are all strictly less than $a$ with respect to the order $<_B$. We observe that $\left|\vec{\beta}\right|<n$ implies that $a$ is $0$-colored, since we have taken the increasing factorization with respect to $<_B$. For ease of notation, we will write
\[B_d=(\underbrace{a,...,a}_{p\text{ times}})\cdot u=a^p\cdot u,\]
where it is understood that $a$ is $0$-colored, and the superscript $p$ means that the letter $a$ is repeated $p$ times.


\vspace{.25cm}
\textbf{Case 1, $\gamma$}\hspace{.5cm}(Case 1.1, $\gamma^{-1}$)

$B_d=a^pc$ where $a$ is unbarred and $c\in\mathcal{B}$. Since the banner rules in Definition \ref{defn 4.3} require that $\left|c\right|\geq\left|a\right|$, this can only happen if $c$ has positive color, say $c$ has color $m>0$. Then set
\[\gamma(B):=(B',(\omega,b,m)),\]
where
\[B'=:B_1\cdot ...\cdot B_{d-1},\]
\[\omega:=(\underbrace{\left|a\right|,...,\left|a\right|}_{p\text{ times}},\left|c\right|),
\text{ and }b:=0.\]

For example if
\[B_d=(4^0,4^0,4^0,4^0,9^2),\]
then
\[(\omega,b,m)=((4,4,4,4,9),0,2).\]

\vspace{.25cm}
\textbf{Case 2, $\gamma$}

$B_d=a^pc,i_1,i_2,...i_l$ where $a$ is unbarred, $c\in\mathcal{B}$, and $i_1$ is unbarred. Again since $a$ is unbarred, $c$ must have positive color. Next, we find the index $s$ such that $1\leq s\leq l$ and either one of the following subcases hold:

\vspace{.25cm}
\textbf{Case 2.1, $\gamma$}\hspace{.5cm}(Case 1.1, $\gamma^{-1}$)

$i_1\leq_B i_2\leq_B ...\leq_B i_l$ are all $0$-colored and unbarred. We then take $s=l$ and set
\[\gamma(B):=(B',(\omega,b,m)),\]
where $c$ has color $m>0$, and where
\[B'=:B_1\cdot ...\cdot B_{d-1},\]
\[\omega:=(\left|i_1\right|,...,\left|i_l\right|,\underbrace{\left|a\right|,...,\left|a\right|}_{p\text{ times}},\left|c\right|),
\text{ and } b:=0.\]

For example if
\[B_d=(4^0,4^0,9^1,2^0,2^0,3^0),\]
then
\[(\omega,b,m)=((2,2,3,4,4,9),0,1).\]

\vspace{.25cm}
\textbf{Case 2.2, $\gamma$}\hspace{.5cm}(Case 4.2, $\gamma^{-1}$)

$i_1\leq_B ...\leq_B i_{s-1}$ are all $0$-colored and unbarred while $i_s$ is barred. Then set
\[\gamma(B):=(B',(\omega,b,0)),\]
where
\[B'=:B_1\cdot ...\cdot B_{d-1}\cdot\tilde{B_d},\]
\[\tilde{B_d}:=a^pci_{s+1}...i_l,\]
\[\omega=:(\left|i_1\right|,...,\left|i_s\right|),
\text{ and }b:=1.\]

For example if
\[B_d=(5^0,5^0,8^3,1^0,4^0,\bar{4^0},2^0,7^1),\]
then
\[\tilde{B_d}=(5^0,5^0,8^3,2^0,7^1),\]
\[(\omega,b,0)=((1,4,4),1,0).\]

\vspace{.25cm}
\textbf{Case 2.3, $\gamma$}\hspace{.5cm}(Cases 1.2 and 2.1, $\gamma^{-1}$)

$i_1\leq_B ...\leq_B i_{s-1}$ are all $0$-colored and unbarred while $i_s$ is positively colored, say $i_s$ has color $m>0$. Then set
\[\gamma(B):=(B',(\omega,b,m)),\]
where
\[B'=:B_1\cdot ...\cdot B_{d-1}\cdot\tilde{B_d},\]
\[\tilde{B_d}:=a^pci_{s+1}...i_l,\]
\[\omega:=(\left|i_1\right|,...,\left|i_s\right|),
\text{ and }b:=0.\]

For example if
\[B_d=(5^0,5^0,8^1,1^0,4^0,7^3,6^2),\]
then
\[\tilde{B_d}=(5^0,5^0,8^1,6^2),\]
\[(\omega,b,m)=((1,4,7),0,3).\]

\vspace{.25cm}
\textbf{Case 3, $\gamma$}

$B_d=a^pc,i_1,i_2,...i_l$ where $a$ is unbarred, $c\in\mathcal{B}$, and $i_1$ is barred. Again this implies $c$ must have positive color. First, find the index $r$ such that $i_1\geq_B ...\geq_B i_{r-1}$ are all barred while $i_r$ is unbarred (note $1<r\leq l$). Then find the index $s$ such that $r\leq s\leq l$ and either one of the following subcases hold:

\vspace{.25cm}
\textbf{Case 3.1, $\gamma$}\hspace{.5cm}(Case 4.3, $\gamma^{-1}$)

$i_r\leq_B ...\leq_B i_s$ are all $0$-colored, unbarred, and $\left|i_s\right|\leq\left|i_{r-1}\right|$, while $\left|i_{s+1}\right|>\left|i_{r-1}\right|$ or $s=l$. Then set
\[\gamma(B):=(B',(\omega,b,0)),\]
where
\[B':=B_1\cdot ...\cdot B_{d-1}\cdot\tilde{B_d},\]
\[\tilde{B_d}:=a^pci_{s+1}...i_l,\]
\[\omega:=(\left|i_r\right|,\left|i_{r+1}\right|,...,\left|i_s\right|,\left|i_{r-1}\right|,\left|i_{r-2}\right|,...,\left|i_1\right|),
\text{ and }b:=r-1.\]

For example if
\[B_d=(4^0,4^0,6^1,\bar{3^0},\bar{3^0},\bar{2^0},1^0,2^0,3^0),\]
then
\[\tilde{B_d}=(4^0,4^0,6^1,3^0),\]
\[(\omega,b,0)=((1,2,2,3,3),3,0).\]

\vspace{.25cm}
\textbf{Case 3.2, $\gamma$}\hspace{.5cm}(Case 4.2, $\gamma^{-1}$)

$i_r\leq_B ...\leq_B i_{s-1}$ are all $0$-colored, unbarred, and $\left|i_{s-1}\right|\leq\left|i_{r-1}\right|$, while $i_s$ is barred and $\left|i_{s}\right|\leq\left|i_{r-1}\right|$. Then set
\[\gamma(B):=(B',(\omega,b,0)),\]
where
\[B'=:B_1\cdot ...\cdot B_{d-1}\cdot\tilde{B_d},\]
\[\tilde{B_d}:=a^pci_{s+1}...i_l,\]
\[\omega:=(\left|i_r\right|,\left|i_{r+1}\right|,...,\left|i_s\right|,\left|i_{r-1}\right|,\left|i_{r-2}\right|,...,\left|i_1\right|),
\text{ and }b:=r.\]

For example if
\[B_d=(6^0,6^0,8^2,\bar{5^0},\bar{4^0},\bar{4^0},2^0,4^0,\bar{4^0},1^0,9^1),\]
then
\[\tilde{B_d}=(6^0,6^0,8^2,1^0,9^1),\]
\[(\omega,b,0)=((2,4,4,4,4,5),4,0).\]

\vspace{.25cm}
\textbf{Case 3.3, $\gamma$}\hspace{.5cm}(Case 3.2, $\gamma^{-1}$)

$i_r\leq_B ...\leq_B i_{s-1}$ are all $0$-colored, unbarred, and $\left|i_{s-1}\right|\leq\left|i_{r-1}\right|$, while $i_s$ is positively colored, say $i_s$ has color $m>0$, and $\left|i_{s}\right|\leq\left|i_{r-1}\right|$. Then set
\[\gamma(B):=(B',(\omega,b,m)),\]
where $c$ has color $m>0$, and where
\[B':=B_1\cdot ...\cdot B_{d-1}\cdot\tilde{B_d},\]
\[\tilde{B_d}:=a^pci_{s+1}...i_l,\]
\[\omega:=(\left|i_r\right|,\left|i_{r+1}\right|,...,\left|i_s\right|,\left|i_{r-1}\right|,\left|i_{r-2}\right|,...,\left|i_1\right|),
\text{ and }b:=r-1.\]

For example if
\[B_d=(5^0,8^1,\bar{4^0},\bar{2^0},1^0,1^1,1^3),\]
then
\[\tilde{B_d}=(5^0,8^1,1^3),\]
\[(\omega,b,m)=((1,1,2,4),2,1).\]

\vspace{.25cm}
\textbf{Case 4, $\gamma$}

$B_d=a^pi_1,i_2,...i_l$ where $a$ is barred and $i_1$ is unbarred. Then find the index $s$ such that $1\leq s\leq l$ and either one of the following subcases hold:

\vspace{.25cm}
\textbf{Case 4.1, $\gamma$}\hspace{.5cm}(Case 4.1, $\gamma^{-1}$)

$i_1\leq_B ...\leq_B i_l$ are all $0$-colored and unbarred, so we take $s=l$ and set
\[\gamma(B):=(B',(\omega,b,0)),\]
where
\[B':=B_1\cdot ...\cdot B_{d-1},\]
\[\omega:=(\left|i_1\right|,\left|i_2\right|,...,\left|i_l\right|,\underbrace{\left|a\right|,...,\left|a\right|}_{p\text{ times}}),
\text{ and }b:=p.\]

For example if
\[B_d=(\bar{5^0},\bar{5^0},\bar{5^0},1^0,3^0,5^0),\]
then
\[(\omega,b,0)=((1,3,5,5,5,5),3,0).\]

\vspace{.25cm}
\textbf{Case 4.2, $\gamma$}\hspace{.5cm}(Case 4.4, $\gamma^{-1}$)

$i_1\leq_B ...\leq_B i_{s-1}$ are all $0$-colored and unbarred while $i_s$ is barred. Then set
\[\gamma(B):=(B',(\omega,b,0)),\]
where
\[B':=B_1\cdot ...\cdot B_{d-1}\cdot\tilde{B_d},\]
\[\tilde{B_d}:=a^pi_{s+1}...i_l,\]
\[\omega:=(\left|i_1\right|,\left|i_2\right|,...,\left|i_s\right|),
\text{ and }b:=1.\]

For example if
\[B_d=(\bar{5^0},\bar{5^0},\bar{5^0},1^0,3^0,\bar{4^0},3^1,7^2),\]
then
\[\tilde{B_d}=(\bar{5^0},\bar{5^0},\bar{5^0},3^1,7^2),\]
\[(\omega,b,0)=((1,3,4),1,0).\]

\vspace{.25cm}
\textbf{Case 4.3, $\gamma$}\hspace{.5cm}(Cases 1.3 and 2.2, $\gamma^{-1}$)

$i_1\leq_B ...\leq_B i_{s-1}$ are all $0$-colored and unbarred while $i_s$ is positively colored, say $i_s$ has color $m>0$, and $\left|i_{s+1}\right|\leq\left|a\right|$. Then set
\[\gamma(B):=(B',(\omega,b,m)),\]
where
\[B':=B_1\cdot ...\cdot B_{d-1}\cdot\tilde{B_d},\]
\[\tilde{B_d}:=a^pi_{s+1}...i_l,\]
\[\omega:=(\left|i_1\right|,\left|i_2\right|,...,\left|i_s\right|),
\text{ and }b:=0.\]

For example if
\[B_d=(\bar{6^0},\bar{6^0},2^0,2^0,4^0,8^1,\bar{5^0},4^2),\]
then
\[\tilde{B_d}=(\bar{6^0},\bar{6^0},\bar{5^0},4^2),\]
\[(\omega,b,m)=((2,2,4,8),0,1).\]

\vspace{.25cm}
\textbf{Case 4.4, $\gamma$}\hspace{.5cm}(Case 3.1, $\gamma^{-1}$)

$i_1\leq_B ...\leq_B i_{l-1}$ are all $0$-colored and unbarred while $i_l$ is positively colored, say $i_l$ has color $m>0$, and $\left|i_l\right|\leq\left|a\right|$. Take $s=l$ and set
\[\gamma(B):=(B',(\omega,b,m)),\]
where
\[B':=B_1\cdot ...\cdot B_{d-1},\]
\[\omega:=(\left|i_1\right|,\left|i_2\right|,...,\left|i_l\right|,\underbrace{\left|a\right|,...,\left|a\right|}_{p\text{ times}}),
\text{ and }b:=p.\]

For example if
\[B_d=(\bar{6^0},\bar{6^0},2^0,2^0,4^0,5^3),\]
then
\[(\omega,b,m)=((2,2,4,5,6,6),2,3).\]

\vspace{.25cm}
\textbf{Case 4.5, $\gamma$}\hspace{.5cm}(Case 2.3, $\gamma^{-1}$)

$i_1\leq_B ...\leq_B i_{l-1}$ are all $0$-colored and unbarred while $i_l$ is positively colored, say $i_l$ has color $m>0$, and $\left|i_l\right|>\left|a\right|$. Take $s=l$ and set
\[\gamma(B):=(B',(\omega,b,m)),\]
where
\[B':=B_1\cdot ...\cdot B_{d-1}\cdot\tilde{B_d},\]
\[\tilde{B_d}:=a^pi_1...i_{l-1},\]
\[\omega:=(\left|i_l\right|),
\text{ and }b:=0.\]

For example if
\[B_d=(\bar{6^0},\bar{6^0},2^0,2^0,4^0,7^3),\]
then
\[\tilde{B_d}=(\bar{6^0},\bar{6^0},2^0,2^0,4^0),\]
\[(\omega,b,m)=((7),0,3).\]

\vspace{.25cm}
\textbf{Case 4.6, $\gamma$}\hspace{.5cm}(Case 2.3, $\gamma^{-1}$)

$i_1\leq_B ...\leq_B i_{s-1}$ are all $0$-colored and unbarred while $i_s, i_{s+1}$ are both positively colored, say $i_{s+1}$ has color $m>0$, and $\left|i_{s+1}\right|>\left|a\right|$. Then set
\[\gamma(B):=(B',(\omega,b,m)),\]
where
\[B':=B_1\cdot ...\cdot B_{d-1}\cdot\tilde{B_d},\]
\[\tilde{B_d}:=a^pi_1...i_si_{s+2}i_{s+3}...i_{l},\]
\[\omega:=(\left|i_{s+1}\right|),
\text{ and }b:=0.\]

For example if
\[B_d=(\bar{6^0},\bar{6^0},2^0,2^0,4^0,7^3,8^2),\]
then
\[\tilde{B_d}=(\bar{6^0},\bar{6^0},2^0,2^0,4^0,7^3),\]
\[(\omega,b,m)=((8),0,2).\]

\vspace{.25cm}
\textbf{Case 5, $\gamma$}

$B_d=a^pi_1,i_2,...i_l$ where $a$ and $i_1$ are barred. First, find the index $r$ such that $i_1\geq_B ...\geq_B i_{r-1}$ are all barred while $i_r$ is unbarred (note $1<r\leq l$). Then find the index $s$ such that $r\leq s\leq l$ and either one of the following subcases hold:

\vspace{.25cm}
\textbf{Case 5.1, $\gamma$}\hspace{.5cm}(Case 4.1, $\gamma^{-1}$)

$i_r\leq_B ...\leq_B i_l$ are all $0$-colored, unbarred, and $\left|i_l\right|\leq\left|i_{r-1}\right|$. Then we take $s=l$ and set
\[\gamma(B):=(B',(\omega,b,0)),\]
where
\[B':=B_1\cdot ...\cdot B_{d-1},\]
\[\omega:=((\left|i_r\right|,\left|i_{r+1}\right|,...,\left|i_l\right|,\left|i_{r-1}\right|,\left|i_{r-2}\right|,...,\left|i_1\right|,\underbrace{\left|a\right|,...,\left|a\right|}_{p\text{ times}}),
\text{ and }b:=p+r-1.\]

For example if
\[B_d=(\bar{7^0},\bar{7^0},\bar{6^0},\bar{4^0},1^0,2^0,4^0),\]
then
\[(\omega,b,0)=((1,2,4,4,6,7,7),4,0).\]

\vspace{.25cm}
\textbf{Case 5.2, $\gamma$}\hspace{.5cm}(Case 4.5, $\gamma^{-1}$)

$i_r\leq_B ...\leq_B i_s$ are all $0$-colored, unbarred, $\left|i_s\right|\leq\left|i_{r-1}\right|$, and $\left|i_{r-1}\right|<\left|i_{s+1}\right|\leq\left|a\right|$. Then set
\[\gamma(B):=(B',(\omega,b,0)),\]
where
\[B':=B_1\cdot ...\cdot B_{d-1}\cdot\tilde{B_d},\]
\[\tilde{B_d}:=a^pi_{s+1}...i_l,\]
\[\omega:=((\left|i_r\right|,\left|i_{r+1}\right|,...,\left|i_s\right|,\left|i_{r-1}\right|,\left|i_{r-2}\right|,...,\left|i_1\right|),
\text{ and }b:=r-1.\]

For example if
\[B_d=(\bar{7^0},\bar{7^0},\bar{6^0},\bar{4^0},1^0,2^0,5^0,8^2,1^0),\]
then
\[\tilde{B_d}:=(\bar{7^0},\bar{7^0},5^0,8^2,1^0),\]
\[(\omega,b,0)=((1,2,4,6),2,0).\]

\vspace{.25cm}
\textbf{Case 5.3, $\gamma$}\hspace{.5cm}(Case 2.4, $\gamma^{-1}$)

$i_r\leq_B ...\leq_B i_s$ are all $0$-colored, unbarred, $\left|i_s\right|\leq\left|i_{r-1}\right|$, and $i_{s+1}$ is positively colored, say $i_{s+1}$ has color $m>0$, with $\left|i_{s+1}\right|>\left|a\right|$. If $\left|i_{r-1}\right|\geq\left|i_{s+2}\right|$, then set
\[\gamma(B):=(B',(\omega,b,m)),\]
where
\[B':=B_1\cdot ...\cdot B_{d-1}\cdot\tilde{B_d},\]
\[\tilde{B_d}:=a^pi_1...i_{r-2}i_ri_{r+1}...i_si_{r-1}i_{s+2}i_{s+3}...i_l,\]
\[\omega:=(\left|i_{s+1}\right|),
\text{ and }b:=0.\]

For example if
\[B_d=(\bar{7^0},\bar{7^0},\bar{6^0},\bar{4^0},1^0,2^0,8^1,1^0),\]
then
\[\tilde{B_d}=(\bar{7^0},\bar{7^0},\bar{6^0},1^0,2^0,\bar{4^0},1^0),\]
\[(\omega,b,m)=((8),0,1).\]

\vspace{.25cm}
\textbf{Case 5.4, $\gamma$}\hspace{.5cm}(Case 2.3, $\gamma^{-1}$)

$i_r\leq_B ...\leq_B i_s$ are all $0$-colored, unbarred, $\left|i_s\right|\leq\left|i_{r-1}\right|$, and $i_{s+1}$ is positively colored, say $i_{s+1}$ has color $m>0$, with $\left|i_{s+1}\right|>\left|a\right|$. If $\left|i_{r-1}\right|<\left|i_{s+2}\right|$ or if $s+1=l$, then set
\[\gamma(B):=(B',(\omega,b,m)),\]
where
\[B':=B_1\cdot ...\cdot B_{d-1}\cdot\tilde{B_d},\]
\[\tilde{B_d}:=a^pi_1...i_si_{s+2}i_{s+3}...i_l,\]
\[\omega:=(\left|i_{s+1}\right|),
\text{ and }b:=0.\]

For example if
\[B_d=(\bar{7^0},\bar{7^0},\bar{6^0},\bar{4^0},1^0,2^0,8^1,5^0),\]
then
\[\tilde{B_d}=(\bar{7^0},\bar{7^0},\bar{6^0},\bar{4^0},1^0,2^0,5^0),\]
\[(\omega,b,m)=((8),0,1).\]

\vspace{.25cm}
\textbf{Case 5.5, $\gamma$}\hspace{.5cm}(Case 4.4, $\gamma^{-1}$)

$i_r\leq_B ...\leq_B i_{s-1}$ are all $0$-colored and unbarred while $i_s$ is barred with $\left|i_s\right|\leq\left|i_{r-1}\right|$. Then set
\[\gamma(B):=(B',(\omega,b,0)),\]
where
\[B':=B_1\cdot ...\cdot B_{d-1}\cdot\tilde{B_d},\]
\[\tilde{B_d}:=a^pi_{s+1}...i_l,\]
\[\omega:=(\left|i_r\right|,\left|i_{r+1}\right|,...,\left|i_s\right|,\left|i_{r-1}\right|,\left|i_{r-2}\right|,...,\left|i_1\right|),
\text{ and }b:=r.\]

For example if
\[B_d=(\bar{6^0},\bar{6^0},\bar{5^0},1^0,2^0,2^0,\bar{4^0},\bar{2^0},2^0),\]
then
\[\tilde{B_d}=(\bar{6^0},\bar{6^0},\bar{2^0},2^0),\]
\[(\omega,b,0)=((1,2,2,4,5),2,0).\]

\vspace{.25cm}
\textbf{Case 5.6, $\gamma$}\hspace{.5cm}(Case 3.1, $\gamma^{-1}$)

$i_r\leq_B ...\leq_B i_{l-1}$ are all $0$-colored and unbarred while $i_l$ is positively colored, say $i_l$ has color $m>0$, with $\left|i_l\right|\leq\left|i_{r-1}\right|$. Then take $s=l$ and set
\[\gamma(B):=(B',(\omega,b,m)),\]
where
\[B':=B_1\cdot ...\cdot B_{d-1},\]
\[\omega:=(\left|i_r\right|,\left|i_{r+1}\right|,...,\left|i_l\right|,\left|i_{r-1}\right|,\left|i_{r-2}\right|,...,\left|i_1\right|,\underbrace{\left|a\right|,...,\left|a\right|}_{p\text{ times}}),
\text{ and }b:=p+r-1.\]

For example if
\[B_d=(\bar{6^0},\bar{6^0},\bar{5^0},1^0,2^0,2^0,4^5),\]
then
\[(\omega,b,m)=((1,2,2,4,5,6,6),3,5).\]

\vspace{.25cm}
\textbf{Case 5.7, $\gamma$}\hspace{.5cm}(Case 3.3, $\gamma^{-1}$)

$i_r\leq_B ...\leq_B i_{s-1}$ are all $0$-colored and unbarred while $i_s$ is positively colored, say $i_s$ has color $m>0$, with $\left|i_s\right|\leq\left|i_{r-1}\right|$. If $\left|i_{s+1}\right|\leq\left|a\right|$, then set
\[\gamma(B):=(B',(\omega,b,m)),\]
where
\[B'=:B_1\cdot ...\cdot B_{d-1}\cdot\tilde{B_d},\]
\[\tilde{B_d}:=a^pi_{s+1}...i_l,\]
\[\omega:=(\left|i_r\right|,\left|i_{r+1}\right|,...,\left|i_s\right|,\left|i_{r-1}\right|,\left|i_{r-2}\right|,...,\left|i_1\right|),
\text{ and }b:=r-1.\]

For example if
\[B_d=(\bar{8^0},\bar{8^0},\bar{5^0},1^0,2^0,2^0,4^5,7^3),\]
then
\[\tilde{B_d}=(\bar{8^0},\bar{8^0},7^3),\]
\[(\omega,b,m)=((1,2,2,4,5),1,5).\]

\vspace{.25cm}
\textbf{Case 5.8, $\gamma$}\hspace{.5cm}(Case 2.3, $\gamma^{-1}$)

$i_r\leq_B ...\leq_B i_{s-1}$ are all $0$-colored and unbarred while $i_s$ is positively colored with $\left|i_s\right|\leq\left|i_{r-1}\right|$. If $\left|i_{s+1}\right|>\left|a\right|$, then $i_{s+1}$ must be positively colored, say $i_{s+1}$ has color $m>0$. Then set
\[\gamma(B):=(B',(\omega,b,m)),\]
where
\[B':=B_1\cdot ...\cdot B_{d-1}\cdot\tilde{B_d},\]
\[\tilde{B_d}:=a^pi_1...i_si_{s+2}i_{s+3}...i_l,\]
\[\omega:=(\left|i_{s+1}\right|),
\text{ and }b:=0.\]

For example if
\[B_d=(\bar{8^0},\bar{8^0},\bar{5^0},1^0,2^0,2^0,4^5,9^3),\]
then
\[\tilde{B_d}=(\bar{8^0},\bar{8^0},\bar{5^0},1^0,2^0,2^0,4^5),\]
\[(\omega,b,m)=((9),0,3).\]

\vspace{.25cm}
This completes the description of the map $\gamma$. Next we describe $\gamma^{-1}$. Suppose we are given a banner $B$ with increasing factorization $B=B_1\cdot ...\cdot B_d$ where $B_d=a^pj_1...j_l$, and an $m$-colored marked sequence $(\omega,b,m)$ where $0\leq m\leq N-1$ and $\omega=(\omega_1,...,\omega_q)$. Here the letter $a$ may have any color, and we do not specify its color. For this letter only we use the superscript $p$ to denote that $a$ is repeated $p$ times where $p>0$. 

\vspace{.25cm}
\textbf{Case 1, $\gamma^{-1}$}

Suppose $m>0$, $b=0$, $q>1$, and one of the following subcases hold:

\vspace{.25cm}
\textbf{Case 1.1, $\gamma^{-1}$}\hspace{.5cm}(Cases 1 and 2.1, $\gamma$)

$\omega_{q-1}^0\geq_B a$. If $\omega_{q-1}$ appears $q-1$ or $q$ times in the sequence $\omega$, then set
\[B_{d+1}:=\underbrace{\omega_{q-1}^0...\omega_{q-1}^0}_{q-1 \text{ times}}\omega_q^m.\]
Otherwise, $\omega_{q-1}$ appears $r$ times with $r<q-1$ and we set
\[B_{d+1}:=\underbrace{\omega_{q-1}^0...\omega_{q-1}^0}_{r\text{ times}}\omega_q^m\omega_1^0\omega_2^0...\omega_{q-r-1}^0.\]
In either case set
\[\gamma^{-1}\left(B,(\omega,b,m)\right):=B_1\cdot ...\cdot B_d\cdot B_{d+1}.\]

For an example of this case (and for most of the cases below), refer to corresponding case of $\gamma$.

\vspace{.25cm}
\textbf{Case 1.2, $\gamma^{-1}$}\hspace{.5cm}(Case 2.3, $\gamma$)

$\omega_{q-1}^0<_B a$, and $a$ is unbarred. Then set
\[\tilde{B_d}:=a^pj_1\omega_1^0\omega_2^0...\omega_{q-1}^0\omega_q^mj_2...j_l,\]
\[\gamma^{-1}\left(B,(\omega,b,m)\right):=B_1\cdot ...\cdot B_{d-1}\cdot\tilde{B_d}.\]

\vspace{.25cm}
\textbf{Case 1.3, $\gamma^{-1}$}\hspace{.5cm}(Case 4.3, $\gamma$)

$\omega_{q-1}^0<_B a$, and $a$ is barred. Then set
\[\tilde{B_d}:=a^p\omega_1^0\omega_2^0...\omega_{q-1}^0\omega_q^mj_1j_2...j_l,\]
\[\gamma^{-1}\left(B,(\omega,b,m)\right):=B_1\cdot ...\cdot B_{d-1}\cdot\tilde{B_d}.\]

\vspace{.25cm}
\textbf{Case 2, $\gamma^{-1}$}

Suppose $m>0$, $b=0$, $q=1$, and one of the following subcases hold:

\vspace{.25cm}
\textbf{Case 2.1, $\gamma^{-1}$}\hspace{.5cm}(Case 2.3, $\gamma$)

$a$ is unbarred. Then set
\[\tilde{B_d}:=a^pj_1\omega_1^mj_2...j_l,\]
\[\gamma^{-1}\left(B,(\omega,b,m)\right):=B_1\cdot ...\cdot B_{d-1}\cdot\tilde{B_d}.\]

For example if
\[B_d=(5^0,5^0,5^0,8^1,2^3,4^1),\]
\[(\omega,b,m)=((4),0,2),\]
then
\[\tilde{B_d}=(5^0,5^0,5^0,8^1,4^2,2^3,4^1).\]

\vspace{.25cm}
\textbf{Case 2.2, $\gamma^{-1}$}\hspace{.5cm}(Case 4.3, $\gamma$)

$a$ is barred and $\omega_1\leq\left|a\right|$. Then set
\[\tilde{B_d}:=a^p\omega_1^mj_1j_2...j_l,\]
\[\gamma^{-1}\left(B,(\omega,b,m)\right):=B_1\cdot ...\cdot B_{d-1}\cdot\tilde{B_d}.\]

For example if
\[B_d=(\bar{6^0},\bar{6^0},\bar{4^0},2^2,1^0,3^1),\]
\[(\omega,b,m)=((5),0,3),\]
then
\[\tilde{B_d}=(\bar{6^0},\bar{6^0},5^3,\bar{4^0},2^2,1^0,3^1).\]

\vspace{.25cm}
\textbf{Case 2.3, $\gamma^{-1}$}\hspace{.5cm}(Cases 4.5, 4.6, 5.4, 5.8, $\gamma$)

$a$ is barred, $\omega_1>\left|a\right|$, and we find the index $s$ such that $1\leq s\leq l$ and one of the following subcases hold:

\textbf{Case 2.3.1} $j_1\leq_B ...\leq_B j_l$ are all $0$-colored and unbarred, so we take $s=l$.

\textbf{Case 2.3.2} $j_1\leq_B ...\leq_B j_{s-1}$ are all $0$-colored and unbarred while $j_s$ is positively colored.

\textbf{Case 2.3.3} $j_1\geq_B ...\geq_B j_{r-1}$ are all barred while $j_r\leq_B ...\leq_B j_l$ are all $0$-colored, unbarred, and $|j_l|\leq|j_{r-1}|$. Then take $s=l$.

\textbf{Case 2.3.4} $j_1\geq_B ...\geq_B j_{r-1}$ are all barred while $j_r\leq_B ...\leq_B j_s$ are all $0$-colored, unbarred, $|j_s|\leq|j_{r-1}|,$ and $|j_{s+1}|>|j_{r-1}|$.


\textbf{Case 2.3.5} $j_1\geq_B ...\geq_B j_{r-1}$ are all barred while $j_r\leq_B ...\leq_B j_{s-1}$ are all $0$-colored, unbarred, and $j_s$ is positively colored with $|j_s|\leq|j_{r-1}|$.

Once the index $s$ is found, we set
\[\tilde{B_d}:=a^pj_1j_2...j_s\omega_1^mj_{s+1}j_{s+2}...j_l,\]
and
\[\gamma^{-1}\left(B,(\omega,b,m)\right):=B_1\cdot ...\cdot B_{d-1}\cdot\tilde{B_d}.\]
Note that in Cases 2.3.1-2.3.5, $j_s$ is an unbarred letter, thus the banner rules in Definition \ref{defn 4.3} are not violated.

\vspace{.25cm}
\textbf{Case 2.4, $\gamma^{-1}$}\hspace{.5cm}(Case 5.3, $\gamma$)

$a$ is barred, $\omega_1>\left|a\right|$, and we find the index $s$ such that $1\leq s\leq l$ and one of the following subcases hold:


\textbf{Case 2.4.1} $j_1\leq_B ...\leq_B j_{s-1}$ are all $0$-colored and unbarred while $j_s$ is barred. Then set

\[\tilde{B_d}:=a^pj_sj_1j_{2}...j_{s-1}\omega_1^mj_{s+1}j_{s+2}...j_l,\]
\[\gamma^{-1}\left(B,(\omega,b,m)\right):=B_1\cdot ...\cdot B_{d-1}\cdot\tilde{B_d}.\]

\textbf{Case 2.4.2} $j_1\geq_B ...\geq_B j_{r-1}$ are all barred while $j_r\leq_B ...\leq_B j_{s-1}$ are all $0$-colored, unbarred, and $j_s$ is barred with $|j_s|\leq|j_{r-1}|$. Then set

\[\tilde{B_d}:=a^pj_1j_2...j_{r-1}j_sj_rj_{r+1}...j_{s-1}\omega_1^mj_{s+1}j_{s+2}...j_l,\]
\[\gamma^{-1}\left(B,(\omega,b,m)\right):=B_1\cdot ...\cdot B_{d-1}\cdot\tilde{B_d}.\]

\vspace{.25cm}
\textbf{Case 3, $\gamma^{-1}$}

Suppose $m>0$, $b>0$, and one of the following subcases hold:

\vspace{.25cm}
\textbf{Case 3.1, $\gamma^{-1}$}\hspace{.5cm}(Cases 4.4 and 5.6, $\gamma$)

$\bar{\omega^0}_q\geq_B a$, then set
\[B_{d+1}:=\bar{\omega^0}_q\bar{\omega^0}_{q-1}...\bar{\omega^0}_{q-b+1}\omega_1^0\omega_2^0...\omega_{q-b-1}^0\omega_{q-b}^m,\]
\[\gamma^{-1}\left(B,(\omega,b,m)\right):=B_1\cdot ...\cdot B_d\cdot B_{d+1}.\]

\vspace{.25cm}
\textbf{Case 3.2, $\gamma^{-1}$}\hspace{.5cm}(Case 3.3, $\gamma$)

$\bar{\omega^0}_q<_B a$, and $a$ is unbarred. Then set
\[\tilde{B_d}:=a^pj_1\bar{\omega^0}_q\bar{\omega^0}_{q-1}...\bar{\omega^0}_{q-b+1}\omega_1^0\omega_2^0...\omega_{q-b-1}^0\omega_{q-b}^mj_2...j_l,\]
\[\gamma^{-1}\left(B,(\omega,b,m)\right):=B_1\cdot ...\cdot B_{d-1}\cdot\tilde{B_d}.\]

\vspace{.25cm}
\textbf{Case 3.3, $\gamma^{-1}$}\hspace{.5cm}(Case 5.7, $\gamma$)

$\bar{\omega^0}_q<_B a$, and $a$ is barred. Then set
\[\tilde{B_d}:=a^p\bar{\omega^0}_q\bar{\omega^0}_{q-1}...\bar{\omega^0}_{q-b+1}\omega_1^0\omega_2^0...\omega_{q-b-1}^0\omega_{q-b}^mj_1j_2...j_l,\]
\[\gamma^{-1}\left(B,(\omega,b,m)\right):=B_1\cdot ...\cdot B_{d-1}\cdot\tilde{B_d}.\]

\vspace{.25cm}
\textbf{Case 4, $\gamma^{-1}$}

Suppose $m=0$, and one of the following subcases hold:

\vspace{.25cm}
\textbf{Case 4.1, $\gamma^{-1}$}\hspace{.5cm}(Cases 4.1 and 5.1, $\gamma$)

$\bar{\omega^0}_q\geq_B a$, then set
\[B_{d+1}:=\bar{\omega^0}_q\bar{\omega^0}_{q-1}...\bar{\omega^0}_{q-b+1}\omega_1^0\omega_2^0...\omega_{q-b}^0,\]
\[\gamma^{-1}\left(B,(\omega,b,m)\right):=B_1\cdot ...\cdot B_d\cdot B_{d+1}.\]

\vspace{.25cm}
\textbf{Case 4.2, $\gamma^{-1}$}\hspace{.5cm}(Cases 2.2 and 3.2, $\gamma$)

$\bar{\omega^0}_q<_B a$, $a$ is unbarred, and $\omega_{q-b+1}\geq\left|j_2\right|$. Then set
\[\tilde{B_d}:=a^pj_1\bar{\omega^0}_q\bar{\omega^0}_{q-1}...\bar{\omega^0}_{q-b+2}\omega_1^0\omega_2^0...\omega_{q-b}^0\bar{\omega^0}_{q-b+1}j_2...j_l,\]
\[\gamma^{-1}\left(B,(\omega,b,m)\right):=B_1\cdot ...\cdot B_{d-1}\cdot\tilde{B_d}.\]

\vspace{.25cm}
\textbf{Case 4.3, $\gamma^{-1}$}\hspace{.5cm}(Case 3.1, $\gamma$)

$\bar{\omega^0}_q<_B a$, $a$ is unbarred, and $\omega_{q-b+1}<\left|j_2\right|$. Then set
\[\tilde{B_d}:=a^pj_1\bar{\omega^0}_q\bar{\omega^0}_{q-1}...\bar{\omega^0}_{q-b+1}\omega_1^0\omega_2^0...\omega_{q-b}^0j_2...j_l,\]
\[\gamma^{-1}\left(B,(\omega,b,m)\right):=B_1\cdot ...\cdot B_{d-1}\cdot\tilde{B_d}.\]

\vspace{.25cm}
\textbf{Case 4.4, $\gamma^{-1}$}\hspace{.5cm}(Cases 4.2 and 5.5, $\gamma$)

$\bar{\omega^0}_q<_B a$, $a$ is barred, and $\omega_{q-b+1}\geq\left|j_1\right|$. Then set
\[\tilde{B_d}:=a^p\bar{\omega^0}_q\bar{\omega^0}_{q-1}...\bar{\omega^0}_{q-b+2}\omega_1^0\omega_2^0...\omega_{q-b}^0\bar{\omega^0}_{q-b+1}j_1j_2...j_l,\]
\[\gamma^{-1}\left(B,(\omega,b,m)\right):=B_1\cdot ...\cdot B_{d-1}\cdot\tilde{B_d}.\]

\vspace{.25cm}
\textbf{Case 4.5, $\gamma^{-1}$}\hspace{.5cm}(Case 5.2, $\gamma$)

$\bar{\omega^0}_q<_B a$, $a$ is barred, and $\omega_{q-b+1}<\left|j_1\right|$. Then set
\[\tilde{B_d}:=a^p\bar{\omega^0}_q\bar{\omega^0}_{q-1}...\bar{\omega^0}_{q-b+1}\omega_1^0\omega_2^0...\omega_{q-b}^0j_1j_2...j_l,\]
\[\gamma^{-1}\left(B,(\omega,b,m)\right):=B_1\cdot ...\cdot B_{d-1}\cdot\tilde{B_d}.\]

This completes the description of $\gamma^{-1}$. One can check case by case that both maps are well-defined and in fact inverses of each other.
\end{proof}

\section{Recurrence and closed formulas}

In this section we present some recurrence and closed form formulas which are equivalent to Theorems \ref{thm 1.1} and \ref{thm 1.4}.

\begin{cor} \label{rec1}
Let $Q_n(t,r,s)$ denote
\[Q_n(t,r,s):=\sum_{\substack{j\geq 0 \\ \vec{\alpha}\in\mathbb{N}^N \\
\vec{\beta}\in\mathbb{N}^{N-1}}}
Q_{n,j,\vec{\alpha},\vec{\beta}}t^jr^{\vec{\alpha}}s^{\vec{\beta}}.\]

Then for $n\geq 1$, $Q_n(t,r,s)$ satisfies the following recurrence relation
\[Q_n(t,r,s)=\left(
\sum_{\substack{\vec{\mu}\in\mathbb{N}^{N} \\
\vec{\nu}\in\mathbb{N}^{N-1} \\
\left|\vec{\mu}\right|+\left|\vec{\nu}\right|=n}}
(-1)^{\left|\vec{\nu}\right|}
h_{\mu}e_{\nu}r^{\vec{\mu}}\prod_{m=1}^{N-1}s_m^{\nu_m+\mu_m}\right)\]
\[+\sum_{k=0}^{n-1}Q_k(t,r,s)h_{n-k}\left(t\left[n-k-1\right]_t
+\left[n-k\right]_t\left(\sum_{m=1}^{N-1}s_m\right)\right).\]

\end{cor}

\begin{proof}
From Theorem \ref{thm 1.1} and the proof of Corollary \ref{cor 3.6} we have
\[\sum_{n\geq 0}Q_n(t,r,s)z^n
=\frac{-H(r_0z)\left(\prod_{m=1}^{N-1}E(-s_mz)H(r_ms_mz)\right)}
{\sum_{n\geq 0}\left(t\left[n-1\right]_t+\left[n\right]_t
\left(\sum_{m=1}^{N-1}s_m\right)\right)h_nz^n},\]
where $\left[-1\right]_t:=-t^{-1}$. Thus
\[\left(\sum_{n\geq 0}Q_n(t,r,s)z^n\right)
\left(\sum_{n\geq 0}\left(t\left[n-1\right]_t+\left[n\right]_t
\left(\sum_{m=1}^{N-1}s_m\right)\right)h_nz^n\right)\]
\[=-H(r_0z)\left(\prod_{m=1}^{N-1}E(-s_mz)H(r_ms_mz)\right).\]

Next we take the coefficient of $z^n$ on both sides,
\[\sum_{k=0}^{n}Q_k(t,r,s)h_{n-k}\left(t\left[n-k-1\right]_t
+\left[n-k\right]_t\left(\sum_{m=1}^{N-1}s_m\right)\right)\]
\[=\sum_{\substack{\vec{\mu}\in\mathbb{N}^{N} \\
\vec{\nu}\in\mathbb{N}^{N-1} \\
\left|\vec{\mu}\right|+\left|\vec{\nu}\right|=n}}
(-1)^{\left|\vec{\nu}\right|+1}
h_{\mu}e_{\nu}r^{\vec{\mu}}\prod_{m=1}^{N-1}s_m^{\nu_m+\mu_m}.\]

Solving for $Q_n(t,r,s)$ yields the desired recurrence.

\end{proof}

\begin{cor} \label{closed1}

For $n\geq 1$ we have
\[Q_n(t,r,s)=\sum_{l=0}^{n-1}\sum_{\substack{k_0,...,k_l\geq 1 \\
\sum k_i=n}}
P_{k_l}\left(\prod_{j=0}^{l-1}h_{k_j}C_{k_j}\right)
+\left(\prod_{i=0}^{l}h_{k_i}C_{k_i}\right),\]
where
\[P_k=P_k(r,s,\xx):=\sum_{\substack{\vec{\mu}\in\mathbb{N}^{N} \\
\vec{\nu}\in\mathbb{N}^{N-1} \\
\left|\vec{\mu}\right|+\left|\vec{\nu}\right|=k}}
(-1)^{\left|\vec{\nu}\right|}
h_{\mu}e_{\nu}r^{\vec{\mu}}\prod_{m=1}^{N-1}s_m^{\nu_m+\mu_m}\]
and
\[C_k=C_k(t,s):=t\left[k-1\right]_t+\left[k\right]_t\left(\sum_{m=1}^{N-1}s_m\right).\]

\end{cor}

\begin{proof}

Using the above notation, Corollary \ref{rec1} can be written as
\[Q_n(t,r,s)=P_n+\sum_{k=0}^{n-1}Q_{k}(t,r,s)h_{n-k}C_{n-k}\]
\[=P_n+h_nC_n+\sum_{k=1}^{n-1}Q_{k}(t,r,s)h_{n-k}C_{n-k}.\]

And now we show that the right hand side of Corollary \ref{closed1} satisfies the same recurrence. Indeed
\[\hspace{0cm} P_n+h_nC_n+\sum_{k=1}^{n-1}\left[
\sum_{l=1}^{k}\sum_{\substack{k_1,...,k_l\geq 1 \\ \sum k_i=k}}
P_{k_l}\left(\prod_{j=1}^{l-1}h_{k_j}C_{k_j}\right)
+\left(\prod_{i=1}^{l}h_{k_i}C_{k_i}\right)\right]h_{n-k}C_{n-k}\]
\[\hspace{0cm} =P_n+h_nC_n+\sum_{k=1}^{n-1}\left[
\sum_{l=1}^{n-k}\sum_{\substack{k_1,...,k_l\geq 1 \\ \sum k_i=n-k}}
P_{k_l}\left(\prod_{j=1}^{l-1}h_{k_j}C_{k_j}\right)
+\left(\prod_{i=1}^{l}h_{k_i}C_{k_i}\right)\right]h_{k}C_{k}\]
\[=\sum_{l=0}^{n-1}\sum_{\substack{k_0,...,k_l\geq 1 \\
\sum k_i=n}}
P_{k_l}\left(\prod_{j=0}^{l-1}h_{k_j}C_{k_j}\right)
+\left(\prod_{i=0}^{l}h_{k_i}C_{k_i}\right).\]



\end{proof}

Let 
\[A_n^{\maj,\exc,\vec{\fix},\vec{\col}}(q,t,r,s)
:=\sum_{\pi\in C_N\wr S_n}q^{\maj(\pi)}t^{\exc(\pi)}r^{\vec{\fix}(\pi)}
s^{\vec{\col}(\pi)},\]
\[\left[\begin{array}{c} n \\ k \end{array}\right]_q:=
\frac{[n]_q!}{[n-k]_q![k]_q!},\]
\[\left[\begin{array}{c} n \\
k_0,...,k_l \end{array}\right]_q:=
\frac{[n]_q!}{[k_0]_q![k_1]_q!...[k_l]_q!},\]
\[\left[\begin{array}{c} n \\
\vec{\mu},\vec{\nu} \end{array}\right]_q:=
\frac{[n]_q!}{[\mu_0]_q![\mu_1]_q!...[\mu_{N-1}]_q!
[\nu_1]_q![\nu_1]_q!...[\nu_{N-1}]_q!}\]
if $\vec{\mu}\in\mathbb{N}^{N},
\vec{\nu}\in\mathbb{N}^{N-1}$ and $|\vec{\mu}|+|\vec{\nu}|=n$. We now apply the stable principal specialization to Corollaries \ref{rec1} and \ref{closed1} to obtain a recurrence and closed form formula for $A_n^{\maj,\exc,\vec{\fix},\vec{\col}}(q,t,r,s)$.

\begin{cor} \label{rec and closed}

For $n\geq 1$ we have
\[A_n^{\maj,\exc,\vec{\fix},\vec{\col}}(q,t,r,s)\]
\[=\left(\sum_{\substack{\vec{\mu}\in\mathbb{N}^{N} \\
\vec{\nu}\in\mathbb{N}^{N-1} \\
\left|\vec{\mu}\right|+\left|\vec{\nu}\right|=n}}
(-1)^{|\vec{\nu}|}\left[\begin{array}{c} n \\
\vec{\mu},\vec{\nu} \end{array}\right]_q
r^{\vec{\mu}}\left(\prod_{m=1}^{N-1}q^{\nu_m \choose 2}s_m^{\nu_m+\mu_m}
\right)\right)\]
\[+\sum_{k=0}^{n-1}\left[\begin{array}{c} n \\ k \end{array}\right]_q A_k^{\maj,\exc,\vec{\fix},\vec{\col}}(q,t,r,s)\left(
tq[n-k-1]_{tq}+[n-k]_{tq} \sum_{m=1}^{N-1}s_m\right),\]

\vspace{.75cm}
and

\[A_n^{\maj,\exc,\vec{\fix},\vec{\col}}(q,t,r,s)\]
\[=\sum_{l=0}^{n-1}\sum_{\substack{k_0,...,k_l\geq 1 \\ 
\sum k_i=n}}\left(\sum_{\substack{\vec{\mu}\in\mathbb{N}^{N} \\
\vec{\nu}\in\mathbb{N}^{N-1} \\
\left|\vec{\mu}\right|+\left|\vec{\nu}\right|=k_l}}
(-1)^{\left|\vec{\nu}\right|}
\left[\begin{array}{c} n \\
k_0,...,k_{l-1},\vec{\mu},\vec{\nu} \end{array}\right]_q
r^{\vec{\mu}}\prod_{m=1}^{N-1}
q^{{\nu_m\choose 2}}s_m^{\nu_m+\mu_m}\right)\]
\[\hspace{-.7cm} \times\left(\prod_{j=0}^{l-1}tq[k_j-1]_{tq}+[k_j]_{tq}
\sum_{m=1}^{N-1}s_m\right)
+\left[\begin{array}{c} n \\
k_0,...,k_l \end{array}\right]_q\left(\prod_{i=0}^{l}
tq[k_i-1]_{tq}+[k_i]_{tq}\sum_{m=1}^{N-1}s_m\right).\]

\end{cor}

\section{Future work}

In this paper we have generalized the main results of Shareshian and Wachs in \cite{sw}. In Sections 5-7 of \cite{sw}, the authors investigate many other interesting properties exhibited by the Eulerian quasisymmetric functions and the relevant joint distribution formulas. We plan to present the corresponding generalizations of these properties in a forthcoming paper. This includes (as mentioned in Remark \ref{remark1}) a detailed proof that the cv-cycle type colored Eulerian quasisymmetric function $Q_{\check{\lambda},j}$, is in fact a symmetric function.

We expect a further study of $Q_{\check{\lambda},j}$ to be quite fruitful.  In \cite{ssw}, Sagan, Shareshian, and Wachs show that the $q$-analog of the Eulerian numbers and their cycle type refinement introduced in \cite{sw} provide an instance of the cyclic sieving phenomenon (see also \cite{rsw}). We suspect that our colored $q$-analog of the Eulerian numbers and their cycle type refinement will also provide an instance of the cyclic sieving phenomenon. We plan to present such results in a future paper.

\section{Acknowledgments}
I would like to thank my adviser Dr. Michelle Wachs for helping me choose a topic for this paper, as well as her continued advice and support. I am extremely grateful for the time she has spent verifying the accuracy of these results, and for all of her suggestions which have greatly improved the clarity of this paper.


\begin{thebibliography}{xxxx}

\bibitem{abr} R.M. Adin, F. Brenti, Y. Roichman, Descent numbers and major indices for the hyperoctahedral group, Adv. in Appl. Math. 27 (2001) $210-224$.

\bibitem{ar} R.M. Adin, Y. Roichman, The flag major index and group actions on polynomial rings, European J. Combin. 22 (2001) $431-446$.


\bibitem{bb} A. Bj\"orner, F. Brenti, Combinatorics of Coxeter Groups, 
Springer, New York, 2005.

\bibitem{carlitz} L. Carlitz, A combinatorial property of q-Eulerian numbers, Amer. Math. Monthly 82 (1975) $51-54$.

\bibitem{cg} C.-O. Chow, I.M. Gessel, On the descent numbers and major indices for the hyperoctahedral group, Adv. in Appl. Math. 38 (2007) $275-301$. 

\bibitem{dw} J. D\'esarm\'enien, M.L. Wachs, Descent classes of permutations with a given number of fixed points, J. Combin. Theory Ser. A 64 (1993) $311-328$. 

\bibitem{fh2} D. Foata, G.-N. Han, Fix Mahonian calculus III; A quadruple distribution, Monatsh. Math. 154 (2008) $177-197$.

\bibitem{fh} D. Foata, G.-N. Han, Signed words and permutations; a sextuple distribution, Ramanujan J. 19 (2009) $29-52$.

\bibitem{gr} I.M. Gessel, C. Reutenauer, Counting permutations with given cycle structure and descent set, J. Combin. Theory Ser. A 64 (1993) $189-215$.



\bibitem{knuth} D. Knuth, The Art of Computer Programming, Vol. Sorting and Searching, second edition, Addison-Wesley, Reading, MA, 1998.

\bibitem{lothaire} M. Lothaire, Combinatorics on Words, Encyclopedia of Mathematics and Its Applications, Vol. 17, Addison-Wesley, Reading, MA, 1983.

\bibitem{mcmahon} P.A. MacMahon, Combinatory Analysis, 2 volumes, Cambridge University Press, London, 1915-1916. Reprinted by Chelsea, New York, 1960.


\bibitem{reiner} V. Reiner, Signed permutation statistics and cycle type, European J. Combin. 14 (1993) $569-579$.

\bibitem{rsw} V. Reiner, D. Stanton, D. White, The cyclic sieving phenomenon, J. Combin. Theory Ser. A 108 (2004) $17-50$.

\bibitem{ssw} B. Sagan, J. Shareshian, M.L.  Wachs, Eulerian quasisymmetric functions and cyclic sieving, Adv. in Appl. Math. 46 (2011) $536-562$.

\bibitem{sw1} J. Shareshian, M.L. Wachs, q-Eulerian polynomials: excedance number and major index, Electron. Res. Announc. Math. Sci. 13 (2007) $33-45$.


\bibitem{sw} J. Shareshian, M.L. Wachs, Eulerian quasisymmetric functions, Adv. Math. 225 (2010) $2921-2966$.



\bibitem{stanley2} R.P. Stanley, Enumerative Combinatorics, Vol. 2, 1st ed., 
Cambridge University Press, Cambridge, 2001.



\bibitem{wachs} M.L. Wachs, Poset topology: tools and applications, Geometric combinatorics $497-615$, 
Amer. Math. Soc., Providence, RI, 2007.

\end{thebibliography}
\end{document}